\documentclass[12pt,hidelinks]{amsart}

\usepackage{amsmath}
\usepackage{amsthm}
\usepackage{stmaryrd}
\usepackage{amsfonts,mathrsfs,amssymb}
\usepackage{verbatim}
\usepackage{tikz-cd}
\usepackage{times}
\usepackage{fullpage}
\usepackage{enumerate}
\makeatletter
\def\@tocline#1#2#3#4#5#6#7{\relax
  \ifnum #1>\c@tocdepth 
  \else
    \par \addpenalty\@secpenalty\addvspace{#2}%
    \begingroup \hyphenpenalty\@M
    \@ifempty{#4}{%
      \@tempdima\csname r@tocindent\number#1\endcsname\relax
    }{%
      \@tempdima#4\relax
    }%
    \parindent\z@ \leftskip#3\relax \advance\leftskip\@tempdima\relax
    \rightskip\@pnumwidth plus4em \parfillskip-\@pnumwidth
    #5\leavevmode\hskip-\@tempdima
      \ifcase #1
       \or\or \hskip 1em \or \hskip 2em \else \hskip 3em \fi%
      #6\nobreak\relax
    \hfill\hbox to\@pnumwidth{\@tocpagenum{#7}}\par
    \nobreak
    \endgroup
  \fi}
\makeatother

\title{Positivity and $\mathbf{L^2}$ Extension}
\author{Dror Varolin}
\email{dror@math.stonybrook.edu}
\address{Department of Mathematics \newline \indent Stony Brook University \newline \indent Stony Brook, NY 11794-3651}

\newcommand{\noi}{\noindent}

\newcommand{\cE}{{\mathcal E}}
\newcommand{\cF}{{\mathcal F}}

\newcommand{\cO}{{\mathcal O}}

\newcommand{\cZ}{{\mathcal Z}}

\newcommand{\sB}{{\mathscr B}}
\newcommand{\sC}{{\mathscr C}}

\newcommand{\sH}{{\mathscr H}}
\newcommand{\sI}{{\mathscr I}}

\newcommand{\sL}{{\mathscr L}}

\newcommand{\sR}{{\mathscr R}}

\newcommand{\sT}{{\mathscr T}}

\newcommand{\sV}{{\mathscr V}}

\newcommand{\sX}{{\mathscr X}}

\newcommand{\fF}{{\mathfrak F}}

\newcommand{\fH}{{\mathfrak H}}
\newcommand{\fI}{\mathfrak{I}}
\newcommand{\fJ}{{\mathfrak J}}

\newcommand{\fe}{{\mathfrak e}}
\newcommand{\ff}{{\mathfrak f}}
\newcommand{\fg}{{\mathfrak g}}
\newcommand{\fh}{{\mathfrak h}}

\newcommand{\fp}{{\mathfrak p}}

\newcommand{\vp}{\varphi} 
\newcommand{\ve}{\varepsilon}

\newcommand{\bbC}{{\mathbb C}}
\newcommand{\bbD}{{\mathbb D}}

\newcommand{\bbL}{{\mathbb L}}

\newcommand{\bbN}{{\mathbb N}}
\newcommand{\bbO}{{\mathbb O}}
\newcommand{\bbP}{{\mathbb P}}
\newcommand{\bbQ}{{\mathbb Q}}
\newcommand{\bbR}{{\mathbb R}}

\newcommand{\bbV}{{\mathbb V}}

\renewcommand{\Re}{\operatorname{Re}}

\newcommand{\red}{\hfill $\diamond$}

\newcommand{\re}{\Re}

\newcommand{\di}{\partial}
\newcommand{\dbar}{\bar \partial}

\newcommand{\ii}{\sqrt{\text{\rm -}1}}

\newcommand{\tensor}{\otimes}

\def\XXint#1#2#3{{\setbox0=\hbox{$#1{#2#3}{\int}$} 
\vcenter{\hbox{$#2#3$}}\kern-.5\wd0}}

\usepackage{hyperref}

\begin{document}

\theoremstyle{plain}
\newtheorem{m-thm}{\sc Theorem}
\newtheorem*{s-thm}{\sc Theorem}
\newtheorem{lem}{\sc Lemma}[section]
\newtheorem{thm}[lem]{\sc Theorem}
\newtheorem{prop}[lem]{\sc Proposition}
\newtheorem{cor}[lem]{\sc Corollary}

\theoremstyle{definition}
\newtheorem{conj}[lem]{\sc Conjecture}
\newtheorem*{sconj}{\sc Conjecture}
\newtheorem{prob}[lem]{\sc Open Problem}
\newtheorem{defn}[lem]{\sc Definition}
\newtheorem*{s-defn}{\sc Definition}
\newtheorem{qn}[lem]{\sc Question}
\newtheorem{ex}[lem]{\sc Example}
\newtheorem{rmk}[lem]{\sc Remark}
\newtheorem*{s-ex}{\sc Example}
\newtheorem*{s-qn}{\sc Question}
\newtheorem*{s-prop}{\sc Proposition}
\newtheorem*{s-rmk}{\sc Remark}
\newtheorem*{s-rmks}{\sc Remarks}
\newtheorem{rmks}[lem]{\sc Remarks}
\newtheorem*{ack}{\sc Acknowledgments}
\newtheorem*{notation}{\sc \underline{Notation}}

\renewcommand\thesubsubsection{\thesubsection.\roman{subsubsection}}

\begin{abstract}
We examine the relationship between positivity of a vector bundle E and the problem of $L^2$ extension of holomorphic E-valued forms of top degree. In particular, we show by example that Griffiths positivity is not enough. Though the sufficiency of Nakano positivity of E has been known for some time, we provide another proof along the lines of Berndtsson and Lempert. For such a proof we establish a vector bundle analogue of a well-known result of Berndtsson. This result is of independent interest and should have many other useful applications.
\end{abstract}

\maketitle

\section{Introduction}

The $L^2$ Extension Theorem for holomorphic sections of vector bundles has been of fundamental importance since its original inception by Ohsawa and Takegoshi \cite{ot}.  There are many versions of the result, but perhaps the following version strikes a good balance between generality and relative simplicity.

\begin{thm}\label{OTM}
Let $X$ be an essentially Stein manifold equipped with a K\"ahler metric $g$, let $Z$ be a smooth hypersurface cut out by a holomorphic section $T \in H^0(X, \cO (L_Z))$ of the line bundle $L_Z \to X$ associated to the smooth divisor $Z$, and let $\fg$ be a smooth metric for $L_Z$ such that 
\[
\sup _X |T|^2 e^{-\lambda} \le 1.
\]
Let $E \to X$ be a holomorphic vector bundle with a smooth Hermitian metric $\fh$ such that 
\begin{equation}\label{curv-hyp-OT}
\ii \Theta (\fh) \ge_{\rm Nak} t \delta \ii \Theta (\fg)\tensor {\rm Id}_E \quad \text{for all }t \in [0,1].
\end{equation}
Then for every $f \in H^0(Z, \cO_Z (K_Z \tensor E|_{Z}))$ satisfying
\[
\int _Z |f|^2_{\fh} < +\infty
\]
there exists $F \in H^0(X, \cO _X(K_X \tensor L_Z \tensor E))$ such that 
\[
F|_Z = f\wedge dT  \quad \text{and} \quad \int _X |F|^2_{\fh\tensor \fg} \le \frac{\pi (1+\delta)}{\delta} \int _Z |f|^2_{\fh} .
\]
\end{thm}

Though it is not immediately obvious, Theorem \ref{OTM} follows from \cite[Theorem 2.1]{gz}.  (If the optimal constant $\frac{\pi(1+\delta)}{\delta}$ is not insisted upon then the result is much older.)  However, the theorem of Guan-Zhou, itself an optimal-constant version of a theorem of Ohsawa \cite[Theorem 4]{o5}, is very general, and a large fraction of the content of \cite{gz} is devoted to extracting and proving corollaries.  Theorem \ref{OTM} is not among these corollaries, so in Section \ref{OTM-proofs-section} we show how it follows.  

In the past few years a link has been drawn between $L^2$ Extension and Griffiths positivity.  In fact, in \cite{dnwz} it is shown that \emph{asymptotically universal local $L^2$ extension from a point}\footnote{This is not the  terminology of \cite{dnwz}.} implies Griffiths positivity.  The converse was proved in \cite{kp} for smooth metrics and more generally for metrics whose determinant is locally bounded.

The aforementioned link suggests the following question.

\begin{s-qn}
If the rank of $E$ is at least two, is Theorem \ref{OTM} true if the curvature hypothesis \eqref{curv-hyp-OT} is required to hold in the weakest sense, namely in the sense of Griffiths? 
\end{s-qn}

Unfortunately the answer is negative, as the following theorem shows.

\begin{m-thm}\label{counterexample-thm}
There exist a compact K\"ahler manifold $X$, a compact complex submanifold $Z \subset X$ of codimension $1$, a section $T \in H^0(X, \cO (L_Z))$, a metric $\fg$ for $L_Z$, a holomorphic vector bundle $E \to X$ and a metric $\fh$ for $E \to X$ such that for some $\delta > 0$
\[
\sup _X |T|^2_{\fg}\le 1, \quad  \ii \Theta (\fh)  \ge_{\rm Grif} \delta \ii \Theta (\fg) \quad \text{for all }t \in [0,1],
\]
and the restriction map 
\[
H^0(X, \cO (K_X \tensor L_Z \tensor E)) \ni F \mapsto \frac{F|_Z}{dT}\in  H^0(Z, \cO (K_Z \tensor E|_Z))
\]
fails to be surjective. 
\end{m-thm}

\noi Theorem \ref{counterexample-thm} is proved in Section \ref{examples-section}.

If one nevertheless tries to prove an extension theorem like Theorem \ref{OTM} for a Griffiths-positive vector bundle $E \to X$, the natural step is to pass to the projectivization $\cO _E (1) \to \bbP(E^*)$ and try to invoke the rank-1 case of Theorem \ref{OTM}.  In Section \ref{OT-Grif-pf-section} we investigate this approach.

After the appearance of \cite{gz}, Berndtsson and Lempert gave a proof of Theorem \ref{OTM} for functions on domains and under curvature hypotheses that are slightly more restrictive than \eqref{curv-hyp-OT}.  In Section \ref{OTM-proofs-section} we extend the proof of Berndtsson and Lempert to the setting of vector bundles over Stein manifolds.  To do so, we also need a higher-rank version of Berndtsson's result on positivity of direct images \cite[Theorem 1.1]{bo-annals}, which we now formulate.  

Let $X$ be a relatively compact domain in a complex manifold $Y$ of complex dimension $n$, let $B$ be a complex manifold of dimension $m$, and denote by $p_1 :Y \times B \to Y$ and $p_2 :Y \times B \to B$ the Cartesian projections. Let $\tilde E \to Y$ be a holomorphic vector bundle of rank $r$.  Fix a smooth metric $\tilde \fh$ for the holomorphic vector bundle $p_1^* \tilde E \to Y \times B$.  We set $E := \tilde E|_X$ and $\fh := \tilde \fh |_{X \times B}$.  Thus we have a trivial holomorphic family of vector bundles with non-trivial metric:
\[
(p_1^*E , \fh) \to X\times B \stackrel{p_2}{\to} B.
\]
Though possibly non-trivial, the metric $\fh$ very usefully extends smoothly to $\di X \times B$.

To each $t\in B$ we associate the Bergman space 
\[
\sH _t := \left \{ f \in H^0(X, \cO (K_X \tensor E))\ ;\ \int _X \ii ^{n^2}\left< f \wedge \bar f , \fh^{(t)} \right > < +\infty \right \},
\]
where $\fh ^{(t)}$ is the metric for $E \to X$ that we canonically identify with the metric $\fh (\cdot, t)$ obtained by restricting the metric $\fh$ for $p_1^*E \to X \times B$ to the fiber over $X \times \{t\}$.  It is well-known that $\sH_t$ is a Hilbert space.  Since $X$ is a bounded domain and the metrics $\fh^{(t)}$, $t \in B$, are all smooth up to the boundary of $X$, the vector subspaces $\sH _t \subset H^0(X, \cO (K_X \tensor E))$ do not depend on $t$.  We can therefore define the trivial vector bundle $\sH \to B$ whose fiber over $t \in B$ is $\sH _t$.  Note that a smooth (resp. holomorphic) section of $\sH \to B$ is uniquely determined by a smooth (resp. holomorphic) section of $p_1 ^* E \to X\times B$ whose restriction to each fiber $X \times \{t\}$ lies in $\sH _t$.  

One defines the metric
\[
(f_1,f_2) := \ii ^{n^2}\int _X \left < f_1 \wedge \bar f_2, \fh ^{(t)} \right >, \quad f_1,f_2\in \sH_t
\]
for the possibly infinite-rank vector bundle $\sH \to B$.  Then one has a Chern connection (which is only densely-defined if the rank of $\sH$ is infinite) that we compute in Section \ref{dit-appendix}.  We then have the following theorem.

\begin{m-thm}\label{dit-triv}
With the above notation, assume that $Y$ is Stein and that $X$ is pseudoconvex.  If 
\begin{enumerate}[{\rm a.}]
\item $(E, \fh^{(t)})$ is Nakano-nonnegative on $X$ for each $t \in B$ and 
\item $(E, \fh)$ is $k$-positive (resp. $k$-nonnegative) on $X \times B$
\end{enumerate}
then $\sH \to B$ is $k$-positive (resp. $k$-nonnegative).
\end{m-thm}

\noi The proof of Theorem \ref{dit-triv} occupies Section \ref{dit-appendix}.  

Such a `deformation of Bergman spaces' is a special case of a more general story.  A short summary is discussed in Section \ref{examples-section}, and a more detailed presentation can be found in \cite{v-BLS}.

\begin{ack}
I am grateful to Sebastien Boucksom and Mihai P\u{a}un for stimulating conversations.  This work was done in part during visits to Sorbonne Universit\'e and Universit\"at Bayreuth, and I thank those institutions for their hospitality.  Finally, I am grateful to the anonymous referee for their effort and their very useful comments, which have made the paper much easier to read.
\end{ack}

\section{Background and preliminaries}\label{background-section}

We present some background material for the reader's convenience and to establish notation.

\subsection{Positivity of holomorphic Hermitian vector bundles}\label{smooth-metrics-pos-par}

Let $E \to X$ be a holomorphic vector bundle with a smooth Hermitian metric $\fh$.  The curvature of the Chern connection is a ${\rm Hom}(E,E)$-valued $(1,1)$-form $\Theta (\fh)$ on $X$ that is $\fh$-skew Hermitian, i.e.,  
\[
\fh (\Theta(\fh)v,w) + \fh (v,\Theta(\fh)w) = 0 \quad \text{for all $x \in X$ and all $v, w \in E_x$}.
\]
One therefore obtains a Hermitian form on $T^{1,0} _X \tensor E \to X$ defined on indecomposable tensors by 
\[
\{ \xi \tensor v, \eta \tensor w\}_{\fh, \Theta (\fh)} := \fh (\ii \Theta (\fh)_{\xi \bar \eta} v,w)
\]
and extended to a Hermitian form on $T^{1,0} _X \tensor E$ by the universal property for tensor products.

\begin{defn}\label{pos-defn}
Let $\fh$ be a smooth Hermitian metric.
\begin{enumerate}[i.]
\item For a positive integer $k$, we say that $\fh$ is \emph{$k$-positive} (resp. $k$-negative) if the Hermitian form $\{ \cdot , \cdot \}_{\fh, \Theta (\fh)}$ is positive (resp. negative) on all tensors of rank at most $k$.
\item Griffiths-positive (resp. negative) means $1$-positive (resp. negative).
\item Nakano-positive (resp. negative) means $k$-positive (resp. negative) for all $k > 0$.
\red
\end{enumerate}
\end{defn}

\noi Recall that the rank of a tensor $T \in A \tensor B$ is the rank (i.e., the dimension of the image) of the associated linear map $L_T : B^* \to A$.

Next we recall the interaction between curvature and duality.  Given a holomorphic vector bundle $E \to X$ equipped with a smooth Hermitian metric $\fh$, one has a smooth, fiber-preserving conjugate-linear map $\sR : E \to E^*$, which we call the \emph{Riesz map}, defined by 
\[
\sR (v) w := \fh (w,v).
\]
If we equip $E^*$ with the dual metric $\fh ^*$ then the Riesz map is a fiberwise conjugate-linear isometry, and more precisely satisfies 
\[
\fh ^* (\sR w,\sR v) = \fh (v,w).
\]
Taking a smooth local section $v$ and a holomorphic local section $w$ of $E$, we find that 
\[
\left < \dbar (\sR v),  w\right > = \dbar (\left < \sR (v), w\right > ) = \dbar (\fh (w,v)) = \fh (w, \nabla 'v) = \left < \sR (\nabla ' v), w\right >.
\]
From the computation
\begin{eqnarray*}
\fh ^* (\sR v, \sR \dbar w) + \fh ^*(\sR \nabla 'v, \sR w) &=& \fh (\dbar w, v) + \fh(w, \nabla 'v) =  \dbar \fh (w,v) = \dbar \fh^* (\sR v, \sR w) \\
&=& \fh ^*(\dbar \sR v, \sR w) + \fh ^*(\sR v, \nabla ' \sR w)
\end{eqnarray*}
we therefore have 
\[
\dbar \sR = \sR \nabla ' \quad \text{and} \quad \nabla ' \sR = \sR \dbar.
\]
Hence the curvature $\Theta (\fh) = \dbar \nabla ' + \nabla ' \dbar$ satisfies 
\begin{equation}\label{dual-curvatures-eqn}
\fh (\Theta (\fh) v,w) = \fh ^* (\sR w, \sR \Theta (\fh)v) = \fh ^* (\sR w, \Theta (\fh ^*) \sR v) = -  \fh ^* (\Theta (\fh ^*) \sR w, \sR v).
\end{equation}
We can thus conclude that $\fh$ is Griffiths positive if and only if $\fh ^*$ is Griffiths negative, but we cannot conclude the same for $k$-positivity as soon as $k \ge 2$.

The calculus identity 
\begin{equation}\label{bo-calculation}
\ii \di \dbar \fh (s, \sigma) = \fh (- \ii \Theta (\fh) s, \sigma) + \ii \fh (\nabla s, \nabla \sigma)
\end{equation}
for sections $s, \sigma \in \cO(E)_x$, yields the following analytic characterization of Griffiths positivity.  

\begin{prop}\label{griffiths-test-holo}
Let $E \to X$ be a holomorphic vector bundle with smooth Hermitian metric $\fh$.  Then the following are equivalent.
\begin{enumerate}[{\rm a.}]
\item $\fh$ is Griffiths negative.
\item For every $x \in X$ and every $s \in \cO (E)_x$ the function $\fh (s,s)$ is plurisubharmonic.
\item For every $x \in X$ and every $s \in \cO (E)_x$ the function $\log \fh (s,s)$ is plurisubharmonic.
\end{enumerate}
\end{prop}

\begin{proof}
The equivalence of (b) and (c) follows easily from \cite[Theorem 1.6.3]{hormander-book}.  To show that (a) and (b) are equivalent, take $s= \sigma$ in \eqref{bo-calculation}.  If $\Theta (\fh)$ is Griffiths negative then the right hand side is positive for all $s$, so $\fh(s,s)$ is plurisubharmonic.  Conversely if $\fh(s,s)$ is plurisubharmonic for all local holomorphic sections $s$ then in particular $\fh(s,s)$ is plurisubharmonic for all $s$ satisfying $\nabla s(x) = 0$ for some given $x$.  Since (i) Griffiths negativity/positivity is a pointwise condition and (ii) every vector in $E_x$ can be interpolated by a local holomorphic section with vanishing covariant derivative at $x$, we see that $\fh$ is Griffiths negative.
\end{proof}

In \cite{bo-annals} Berndtsson observed that one can similarly characterize $k$-negativity and $k$-positivity. To do so, we introduce some notation.  For a holomorphic coordinate system $z$ we define the smooth $(n-1, n-1)$-forms 
\begin{equation}\label{bo-forms}
\Upsilon ^{i \bar j}(z) := c_{i \bar j} dz ^1 \wedge ... \wedge \widehat{dz ^i} \wedge ...  \wedge d\bar z ^1 \wedge ... \wedge \widehat{d\bar z ^j} \wedge ... \wedge d\bar z ^i
\end{equation}
with $c_{i \bar j}$ a constant so that $\ii dz ^i \wedge d\bar z^i \wedge \Upsilon ^{i \bar j} = dV(z)$.

\medskip

Hereafter, the complex version of the Einstein summation convention is in force.

\medskip

We begin with the following slight generalization of Berndtsson's observation.

\begin{prop}\label{k-prop}
Let $E \to X$ be a holomorphic vector bundle with smooth Hermitian metric $\fh$.
\begin{enumerate}[{\rm i.}]
\item $\fh$ is $k$-negative if and only if for every local holomorphic coordinate system $(U, z)$ and every $k$-tuple of local holomorphic sections $f_1,..., f_k \in H^0(U, \cO(E))$ of $E$ the $(n,n)$-form 
\[
\ii \di \dbar \fh (f_i , f_j) \wedge \Upsilon ^{i \bar j} (z)
\]
is positive.
\item $\fh$ is $k$-positive at $x \in X$ if and only if for every local holomorphic coordinate system $(U,z)$ and every $k$-tuple of local holomorphic sections $f_1,..., f_k\in H^0(U, \cO(E))$ of $E$ satisfying $\nabla f_i (x) = 0$ the $(n,n)$-form 
\[
- \ii \di \dbar \fh (f_i , f_j) \wedge \Upsilon ^{i \bar j} (z)
\]
is positive at $x$.  
\end{enumerate}
\end{prop}

\begin{proof}
First note that a tensor $T \in T^{1,0}_{X,x} \tensor E_x$ has rank $k$ if and only if there is a coordinate system $(U,z)$ and local holomorphic sections $f_1,...,f_k \in H^0(U, \cO(E))$ such that 
\[
T = \tfrac{\di}{\di z^1} \tensor f_1(x)+ \cdots + \tfrac{\di}{\di z^k} \tensor f_k(x).
\]
From \eqref{bo-calculation} we compute that 
\begin{equation}\label{k-bo-formula}
\ii \di \dbar \fh (f_i, f_j) \Upsilon ^{i \bar j} =\left ( \sum _{i ,j=1} ^k- \fh (\Theta (\fh) _{i \bar j} f_i , f_j) + \fh \left ( \sum _i \nabla _i f_i , \sum _j \nabla _j f_j \right )\right ) dV(z).
\end{equation}
Item (ii) follows easily from \eqref{k-bo-formula} and the already-cited fact that for every $x \in X$, every vector in $E_x$ can be interpolated by a local holomorphic section of $E$ with vanishing covariant derivative at $x$.  

Turning to (i), note that while the second term on the right hand side of \eqref{k-bo-formula} is always positive, the first term on the right hand side of \eqref{k-bo-formula} is positive if and only if $\fh$ is $k$-negative.  Therefore if $\fh$ is $k$-negative then the left hand side of \eqref{k-bo-formula} is positive for all coordinate systems $z$ and all local holomorphic sections $f_1,..., f_k$.  Conversely, the left hand side of \eqref{k-bo-formula} is positive for all coordinate systems z and all local holomorphic sections $f_1,..., f_k$ then in particular it is positive for local holomorphic sections satisfying $\nabla f_i (x) = 0$ at some $x \in X$, so $\fh$ is $k$-negative.
\end{proof}

\begin{cor}\label{cone-observation}
Let $E \to X$ be a holomorphic vector bundle.
\begin{enumerate}[{\rm a.}]
\item The set of $k$-negative smooth metrics for $E$ is a convex cone with respect to addition of metrics.
\item The set of Griffiths positive smooth metrics for $E$ is a convex cone with respect to the addition
\[
\fh _1 \stackrel{*}{+} \fh _2 := \left ( \fh _1 ^* + \fh _2 ^* \right )^*.
\]
\end{enumerate}
\end{cor}

\begin{rmk}\label{no-cone-yo!}
On the other hand, since for $k \ge 2$ duality does not interchange $k$-negative smooth Hermitian metrics and $k$-positive smooth Hermitian metrics, we cannot argue $k$-positivity is preserved under harmonic-convex combinations.  In fact, we don't know if for $k \ge 2$ the notion $k$-positivity is closed under harmonic sum, or for that matter, under any useful additive structure on the space of smooth Hermitian metrics.
\red
\end{rmk}

\subsection{H\"ormander-Skoda-Demailly Estimates}\label{hsd-est-par}

H\"ormander's Theorem on the $L^2$ estimates for solutions of the $\dbar$-equation is among the most widely known results in several complex variables.  These estimates were improved by Skoda, and further extended by Demailly.  However, there are some degenerate cases that are less widely known.  Therefore for the sake of establishing notation, we review the story here.

\subsubsection*{\sc Hermitian Endomorphisms}

Let $X$ be a complex manifold equipped with a Hermitian (i.e., $J$-invariant) Riemannian metric $g$ whose metric form we denote $\omega$, and let $E \to X$ be a holomorphic vector bundle equipped with a smooth Hermitian metric $\fh$.  A smooth Hermitian section $\Omega$ of $\Lambda ^{1,1} _X \tensor {\rm Hom}(E,E) \to X$ acts on an $E$-valued $(p,q)$-form $\alpha$ by the formula given in local coordinates $z = (z^1,..., z^n)$ and a holomorphic frame $\mathbf{e}_1, ..., \mathbf{e}_r$ (with dual frame $\mathbf{e}^{*1},...,\mathbf{e}^{*r}$) as
\begin{equation}\label{endo-action}
\Omega ^{\sharp _g} \left ( \alpha ^{\mu} _{I \bar J} dz ^I \wedge d\bar z ^J \tensor \mathbf{e} _{\mu} \right ) := g^{i \bar \ell} \Omega ^{\mu} _{\nu i \bar j_{\gamma} } \alpha^{\mu}  _{I \bar j_1 ... (\bar \ell) _{\gamma} ... \bar j_q} dz^I \wedge d\bar z ^J \tensor \mathbf{e} _{\mu},
\end{equation}
where $g = g_{i \bar j} dz ^i \cdot d\bar z ^j$, $(g^{i\bar j})$ is the inverse matrix of $(g_{i\bar j})$ and $\Omega = \Omega ^{\mu} _{\nu i\bar j} dz ^i \wedge d\bar z ^j \tensor \mathbf{e}_{\mu} \tensor \mathbf{e} ^{*\nu}$.  The notation $(a)_k$ means that the $k^{\text{th}}$ entry is replaced by $a$.  (Formula \eqref{endo-action} appears in the Bochner-Kodaira Identity; see \cite[Page 63]{siu-harmonic}.)

\begin{defn}\label{pos-endos-defn}
We say that a Hermitian section $\Omega$ of $\Lambda ^{1,1} _X \tensor {\rm Hom}(E,E) \to X$ is \emph{$k$-positive at $x \in X$} if the associated Hermitian form $\{ \cdot, \cdot \}_{\fh, \Omega}$ defined on rank-1 tensors in $T^{1,0}_{X,x} \tensor E_x$ by 
\[
\{ \xi \tensor v, \eta \tensor w\}_{\fh, \Omega} := \fh ( \Omega _{\xi \bar \eta} v, w)
\]
is positive definite on all tensors of Rank $k$.
\red
\end{defn}

\begin{s-rmk}
Hermitian symmetry and positivity of $\Omega$ are independent of the metric $\fh$.
\red
\end{s-rmk}

Demailly proved the following result.

\begin{prop}[\text{\cite[Lemma VII(7.2)]{dem-book}}]\label{dem-pos-calc}
Let $X$ be a complex manifold of dimension $n$ and let $E \to X$ be a holomorphic vector bundle of rank $r$.  Fix integers $q$ and $k$ such that $q \ge 1$ and $k\ge \min (n-q+1,r)$.  If a Hermitian endomorphism $\Omega \in H^0(X, \sC^{\infty} (\Lambda ^{1,1}_X \tensor {\rm Hom}(E,E)))$ is $k$-positive then for any Hermitian Riemannian metric $g$ on $X$ the pointwise Hermitian operator 
\[
\Omega ^{\sharp_g} : \Lambda ^{n,q} _X \tensor E \to \Lambda ^{n,q} _X \tensor E
\]
is positive-definite.
\end{prop}

Together with standard $L^2$ techniques, Proposition \ref{dem-pos-calc} yields the following well-known result.

\begin{thm}[H\"ormander, Skoda, Demailly]\label{skoda-est}
Let $X$ be a complete K\"ahler manifold equipped with a possibly non-complete K\"ahler metric $g$, let $E \to X$ be a holomorphic vector bundle of rank $r$ and let $\fh$ be a smooth Hermitian metric for $E$.  Fix integers $q \ge 1$ and $k \ge \min (n-q+1, r)$, and let $\Omega \in H^0(X, \sC^{\infty} (\Lambda ^{1,1} _X \tensor {\rm Hom} (E,E)))$ be $k$-positive.  If $\ii \Theta (\fh) -\Omega$ is $k$-nonnegative then for any measurable $E$-valued $(n,q)$-form $\beta$ such that 
\[
\dbar \beta = 0 \quad \text{and} \quad \int _X \left <( \Omega ^{\sharp _q})^{-1}\beta, \beta \right > _{\fh , g}< +\infty
\]
there exists a measurable $E$-valued $(n,q-1)$-form $u$ such that 
\[
\dbar u = \beta \quad \text{and} \quad \int _X \left <u,u\right >_{\fh ,g} \le \int _X \left <( \Omega ^{\sharp _q})^{-1}\beta, \beta \right >_{\fh, g}.
\]
\end{thm}

\subsubsection*{\sc The case of $(n,1)$-forms}

Let $X$ be a complete K\"ahler manifold equipped with a (possibly not complete) K\"ahler metric $g$ and let  $E \to X$ be a holomorphic vector bundle equipped with a smooth Hermitian metric $\fh$ having non-negative curvature in the sense of Nakano.  Theorem \ref{skoda-est} (with $\Omega = \Theta (\fh)$) states that for every $\dbar$-closed $K_X \tensor E$-valued $(0,1)$-form $\beta$ such that 
\begin{equation}\label{skoda-h-norm}
\int _X\left < \left (\Theta (\fh) ^{\sharp_g}\right )^{-1} \beta , \beta\right >_{\fh,g} dV_g < +\infty
\end{equation}
there exists a measurable (and smooth if $\beta$ is smooth) section $u$ of $K_X \tensor E \to X$ such that 
\begin{equation}\label{u-est}
\dbar u = \beta \quad \text{and} \quad \int _X |u|^2_{\fh} dV_g  \le  \int _X\left < \left (\Theta (\fh) ^{\sharp_g}\right )^{-1} \beta , \beta\right >_{\fh,g} dV_g.
\end{equation}

Let $\sL_o$ denote the space of all measurable sections $F$ of $K_X \tensor E\to X$ whose $L^2$-norm 
\[
||F|| := \left (\ii ^{n^2}  \int_X \left < F \wedge \bar F , \fh\right > \right )^{1/2}
\]
is finite.  The subspace $\sH_o:= \sL_o\cap H^0(X, \cO(K_X \tensor E))$ of $\sL_o$ is closed, as is $\sH _o ^{\perp}$.  Let 
\[\label{bergman-proj}
P_o : \sL _o \to \sH _o \quad \text{and} \quad N_o := {\rm Id}- P_o : \sL_o \to \sH_o^{\perp}
\]
denote the orthogonal projections.  Evidently if $F \in \sL_o$ then $N_oF$ is orthogonal to $\sH_o$.  Thus $N_oF$ is the minimal solution of the equation $\dbar u = \dbar F$.  If $\beta := \dbar F$ also satisfies \eqref{skoda-h-norm} then the estimate \eqref{u-est} yields the following result.

\begin{thm}\label{hormander-estimate}
Assume that the metric $\fh$ has non-negative curvature in the sense of Nakano.  Then for any $F\in {\rm Domain} (\dbar) \subset \sL_o$ one has the estimate 
\[
\int _X |N_oF|^2_{\fh} \le   \int _X\left <\left (\Theta (\fh) ^{\sharp_g}\right )^{-1}\dbar F , \dbar F \right >_{g,\fh} ,
\]
provided the right hand side is finite.
\end{thm}

\subsection{Tautological vector bundles}

\subsubsection{\sf The tautological line bundle}\label{O1-par}

Let $V$ be a complex vector space of complex dimension $r$.  We denote by $V^*$ the dual complex vector space, and by $\bbP(V^*)$ the projectivization of $V^*$, i.e., the set of $1$-dimensional subspaces of $V^*$.  The tautological line bundle $\pi : \cO (-1) \to \bbP(V^*)$ assigns to each $\ell \in \bbP(V^*)$ the $1$-dimensional vector space $\ell \subset V^*$.  The dual line bundle $\cO (1) := \cO (-1) ^*$ is called the hyperplane line bundle, and one sets $\cO (\pm k) := \cO (\pm 1) ^{\tensor k}$ for all $k \in \bbN$.

 Elements of the space $V^{**} = V$ of linear functions on $V^*$ can be seen as  holomorphic sections of $\cO (1) \to \bbP(V^*)$, the dual of $\cO (-1) \to \bbP(V^*)$.  It is easy to show that conversely every section of $\cO (1) \to \bbP(V^*)$ arises in this way, i.e.,  
 \begin{equation}\label{o1-glob-eq}
 H^0(\bbP(V^*), \cO (1)) = V.
 \end{equation}
If $V$ is equipped with a Hermitian inner product $\fg$ then one has a smooth Hermitian metric $e^{\vp _{\fg}}$ for $\cO (-1) \to \bbP(V^*)$ that assigns to a vector $v \in \cO (-1) _{\ell}\cong \ell \subset V^*$ its $\fg^*$-square length 
\[
|v|^2e^{\vp _{\fg}} := \fg^* (v,v),
\]
where $\fg ^*$ is the dual Hermitian inner product.  The dual metric $e^{-\vp_{\fg}}$ for $\cO (1) \to \bbP(V^*)$ is called the Fubini-Study metric associated to $\fg$, or simply the $\fg$-Fubini-Study metric.  Its curvature $\di \dbar \vp _{\fg}$ induces a K\"ahler metric 
\[
\omega _{\fg} := \ii \di \dbar \vp _{\fg}
\]
on $\bbP(V^*)$, unfortunately also called the ($\fg$-)Fubini-Study metric.  If $\{\mathbf{e_1}, ... , \mathbf{e}_r\}$ is a $\fg^*$-orthonormal basis for $V^*$ then on the affine chart $U_o := \{ [t^i \cdot \mathbf{e}_i]\ ;\ t^1 = 1\}$ we have 
\[
\omega _{\fg}[t^i \mathbf{e}_i] = \ii \di \dbar \log (1+ |t'|^2), \quad [t^i \mathbf{e}_i] \in U_o,
\]
where $t' = (t^2,..., t^r)$.  The form $\omega _{\fg}$ is independent of the choice of orthonormal basis $\{\mathbf{e}_1, ... , \mathbf{e}_r\}$, and hence is globally defined.  However, $\omega _{\fg}$ does depend on $\fg$ to some extent.  

\subsubsection{\sf The universal quotient bundle}\label{UQB-par}

Let $V$ be a finite-dimensional vector space of dimension $r \ge 2$, and let $\fh$ be a Hermitian inner product for $V$.  Over the projective space $\bbP (V^*)$ one has the exact sequence of vector bundles 
\begin{equation}\label{univ-seq}
0 \to \cO (-1) \to \bbV^* \stackrel{q}{\to} Q \to 0,
\end{equation}
where $\bbV^*$ denotes the trivial holomorphic vector bundle over $\bbP (V^*)$ with fiber $V^*$.  For $x \in \bbP (V^*)$, each element $q(v) \in Q_x$ is an equivalence class of vectors $v \in V^*$, where 
\[
q(v) =q(w) \iff v - w \in x.
\]
We equip $\bbV^*$ with the `constant' Hermitian metric induced by $\fh^*$, and give $Q$ the quotient Hermitian metric $\fh _Q$, defined by 
\[
\fh _Q(q(v), q(v)) := \inf \left \{ \fh^* (w,w)\ ;\  q(v)=q(w) \right \}.
\]
Equivalently, each equivalence class $q(v)\in Q_x$ intersects the $\fh^*$-orthogonal complement $x^{\perp}$ of $x$ in $V^*$ at a unique point $v_x$, and 
\[
\left | q(v)\right |^2 = \fh^* (v_x, v_x).
\]
In fact, $v_x = q^{\dagger}(q(v))$, where $q^{\dagger} :Q \to \bbV^*$ is defined by 
\[
\fh^* (q^{\dagger} (q(v)), w) = \fh _Q (q(v), q(w)).
\]
Indeed, $q^{\dagger} :Q \to \bbV^*$ is clearly injective, and if $w \in x$ then 
\[
\fh^* (q^{\dagger}(q(v)), w) = \fh _Q (q(v), q(w)) = \fh _Q (q(v), 0)= 0,
\]
which shows that $q^{\dagger}q(v_x) \in x^{\perp}$.  In other words, $q^{\dagger} : Q \to \bbV^*$ gives the Hermitian splitting of \eqref{univ-seq}.

Finally, we recall without proof the well-known identity
\[
Q\tensor \cO (1) = T_{\bbP (V^*)}.
\]
(See, for example, \cite[Proposition V.15.7, page 279]{dem-book}.)

\subsubsection{\sf Holomorphic families of vector spaces}\label{O1-of-E-par}

One can adapt the constructions of the Paragraph \ref{O1-par} to the setting of locally trivial families of vector spaces.  Let $E \to X$ be a holomorphic vector bundle.  The fiberwise-projectivization of $E \to X$ yields a locally trivial projective fibration $p : \bbP(E^*) \to X$, as well as holomorphic line bundles 
\[
\cO_E(-1) \to \bbP(E^*) \quad \text{and} \quad \cO _E(1) := \cO_E(-1)^*,
\]
defined fiber-by-fiber as above.  There is also a relationship between sections of $E \to X$ and sections of $\cO _E(1) \to \bbP(E^*)$ analogous to the isomorphism \eqref{o1-glob-eq}, which we state as follows.

\begin{prop}\label{section-id-prop}
Let $\pi :E \to X$ be a holomorphic vector bundle.  The correspondence 
\[
\Gamma (X, E) \ni f \mapsto \hat f \in \Gamma (\bbP(E^*), \cO _E(1))
\]
defined by 
\[
\left < \hat f ([v]), v \right > := \left < f( \pi v), v \right >
\]
is a bijection between sections of $E \to X$ and sections of $\cO _E(1) \to \bbP(E^*)$ that are holomorphic along the fibers of $\bbP(E^*) \to X$.  Moreover, the section $\hat f$ is holomorphic on $p ^{-1} (U) \subset \bbP(E^*)$ if and only if $f$ is holomorphic on $\pi ^{-1} (U)$.  In particular, 
\[
H^0(X, \cO (E)) \ni f \longleftrightarrow \hat f \in H^0(\bbP(E^*), \cO _E (1))
\]
is an isomorphism.
\end{prop}

\section{Proof of Theorem \ref{counterexample-thm}}\label{examples-section}

\subsection{Motivation for Theorem \ref{counterexample-thm}:  Deformation of Bergman spaces}

Let $X$ and $B$ be complex manifolds with $B$ connected, let $\pi :X \to B$ be a proper holomorphic submersion of fiber dimension $n$, and let $E \to X$ be a holomorphic vector bundle with smooth Hermitian metric $\fh$.  To each $t \in B$ one can associate the finite-dimensional inner product space  $\sH _t$ whose underlying vector space is $H^0(X_t, \cO (K_{X_t} \tensor E|_{X_t}))$ and whose inner product is 
\[
(f,g) _t := \ii ^{n^2} \int _{X_t} \left < f \wedge \bar g , \fh \right >.
\]
We endow the disjoint union 
\[
\sH := \coprod _{t \in B} \sH _t
\]
with a projection map $\sH \to B$ whose fiber over $t \in B$ is $\sH _t$, and we give this family of vector spaces additional structure by declaring that a section of $\sH \to B$ is simply a section of the $E$-twisted relative canonical bundle $K_{X/B}\tensor E \to X$ that is holomorphic along the fibers of $\pi : X \to B$.  Similarly, we can define sections of $\sH |_U$ for any subset $U \subset B$ by restricting the family $\pi :X \to B$ to $\pi ^{-1} (U)$.  We denote by $H^0(U, \cO(\sH))$ the collection of holomorphic sections of $\sH |_U \to U$.  When it increases clarity, we use gothic letters to denote by $\ff \in H^0(U,\cO (\sH))$ the section of $\sH \to B$ associated to the section $f$ of $K_{X/B} \tensor E \to X$ such of $\iota _{X_t} ^*f \in \sH _t$, which we denote by the corresponding latin letter.  We say that the section $\ff$ is smooth, holomorphic, measurable, etc. if this is the case for the section $f$.    And if $\ff$ is a section of $\sH \to B$ then for $t \in B$ we define 
\[
\iota _t \ff := \iota _{X_t}^* f \in \sH _t.
\]
\begin{defn}\label{def-of-bergs-defn}
The family $\sH \to B$ of Hilbert spaces with the additional structure of holomorphic and smooth sections described above is called the \emph{deformation of Bergman spaces} associated to the holomorphic family of Hermitian holomorphic vector bundles 
\[
(E,\fh) \to X \stackrel{p}{\to} B.
\]
\end{defn}

If for each point $t \in B$ we can find a neighborhood $U\subset B$ containing $t$ and a collection of sections $\ff _1,..., \fh _N \in H^0 (U, \cO (\sH))$ such that 
\[
\{ \iota _tf_1,..., \iota _tf_N\} \subset \sH _t
\]
is a basis, then we can give $\sH \to B$ the structure of a holomorphic vector bundle by using $\{ \iota _tf_1,..., \iota _tf_N\}$ as a local frame over $U$.

In general, deformations of Bergman spaces need not be vector bundles.  Nevertheless, they do carry a natural Hermitian metric as well as a complex structure of sorts.  If the deformation is indeed a vector bundle then there is a notion of curvature that one can attach to it, namely the curvature of the Chern connection of the $L^2$ metric.  

In his article \cite{bo-annals}, Berndtsson considered this scenario in the case in which $E \to X$ is a holomorphic Hermitian line bundle.  He pointed out that by the $L^2$ Extension Theorem $\sH \to B$ is a Hermitian holomorphic line bundle as soon as $E$ is semi-positive.  

In \cite{v-BLS} the author extended Berndtsson's work in two ways.  
\begin{enumerate}[i.]
\item The author provided $\sH \to B$ with a sort of complex analytic structure even when $\sH \to B$ is not a holomorphic vector bundle.  This structure, called an \emph{iBLS structure}, allows one to define an intrinsic notion of curvature of $\sH \to B$, which agrees with the curvature of the Chern connection over any open set $U \subset B$ on which $\sH|_U$ is a holomorphic vector bundle.
\item The author extended Berndtsson's work to the setting in which $E \to X$ can have higher rank.  This situation was previously considered by Liu and Yang (c.f. \cite{ly}), but under the assumption that $\sH \to B$ is already a holomorphic vector bundle.
\end{enumerate}
In particular, the author proved the following theorem.

\begin{thm}[\text{\cite[Theorem 1]{v-BLS}}]\label{dv-pos-dit}
Suppose $\pi :X \to B$ is a K\"ahler family, i.e., there is a closed form $\omega$    on $X$ whose restriction to fibers is K\"ahler.  If the Hermitian holomorphic vector bundle $E \to X$ is $k$-nonnegative (resp. $k$-positive) then so is the associated deformation of Bergman spaces $\sH \to B$.
\end{thm}

If the deformation of Bergman spaces $\sH \to B$ associated to the proper holomorphic Hermitian family of vector bundles is locally trivial then one can deduce Theorem \ref{dv-pos-dit} from the work of Liu-Yang (and more specifically, from \cite[Theorem 1.6]{ly}).  However, to the best of our knowledge it is not known whether the condition of $k$-positivity of $E \to X$ implies local triviality of $\sH$ unless $k \ge \min ({\rm Rank}(E), \dim X)$, i.e., unless $E \to X$ is Nakano-positive.

In the next paragraph we find what is probably the simplest example of a Griffiths positive vector bundle for which $L^2$ Extension fails, thus showing that Berndtsson's approach to proving local triviality in the rank-1 case cannot be easily extended to the higher-rank setting.  

Unfortunately we do not have a good guess as to whether there is a non-locally trivial deformation of Bergman spaces $\sH \to B$ associated to a proper family of holomorphic Hermitian vector bundles $(E, \fh) \to X \stackrel{p}{\to} B$ such that $\fh$ is $k$-positive for some $k \ge 1$.

\begin{rmk}
If the holomorphic submersion $p :X \to B$ is not proper then the resulting deformation of Bergman spaces is too wild to study without further assumptions.  Theorem \ref{dit-triv} considers a very special but rather useful case in which the resulting family of Bergman spaces is a trivial vector bundle with non-trivial $L^2$ metric.  There is also very interesting work on the non-proper case by \cite{wang} and \cite{pranav} in the case of a proper family of manifolds-with-boundary.
\red
\end{rmk}

\subsection{Curvature of the universal quotient bundle}

In Paragraph \ref{UQB-par} we recalled the definition of the universal quotient bundle $Q \to \bbP(V^*)$, the Hermitian holomorphic vector bundle defined by the short exact sequence of holomorphic Hermitian vector bundles 
\[
0 \to \cO (-1) \to \bbV \to Q \to 0,
\]
where $\bbV \to \bbP(V^*)$ is the trivial Hermitian vector bundle with fiber $V^*$ and constant Hermitian inner product defined by some Hermitian inner product $\fh$ on $V$.  

Let us compute the curvature of the quotient metric $\fh_Q$ for $Q$ at a point $x_o \in \bbP(V^*)$.  We can choose an $\fh^*$-orthonormal basis $e_0,..., e_{r-1}$ for $V^*$ such that $e_1 \in x_o$.  In the open set
\[
U_{x_o} \ni x := [e_0+ z^1e_1 + \cdots + z^{r-1} e_{r-1}] \stackrel{\cong}{\mapsto} z(x) = (z^1,..., z^{r-1}) \in \bbC ^{r-1},
\]
which has $x_o$ as its origin, we take the holomorphic frame for $\cO (-1)|_{U_{x_o}}$ given by 
\[
\ve _0(x) = e_0 + z(x) \cdot \underbar{e},
\]
where $\underbar{e} = (e_1,..., e_{r-1})$, and we let $\fe_1,..., \fe_{r-1}$ be the holomorphic frame for $Q|_{U_{x_o}}$ given by 
\[\label{E-frame-defn}
\fe_j = q(e_j), \quad e_j \in V^* \cong H^0(\bbP(V^*), \cO (\bbV^*)), \ 1 \le j \le r-1.
\]
Writing $q ^{\dagger}(\fe_j) := \zeta _j \ve _0 (x) + e_j$, we have
\[
0 = \fh^* (q ^{\dagger}(\fe_j), \ve _0(x)) = \fh^* ( \zeta _j e_0 + \zeta _j z(x) \cdot \underbar{e} + e_j, e_0 + z(x) \cdot \underbar{e} ) = \bar z^j + \zeta _j (1+ |z|^2),
\]
so that 
\[
q^{\dagger}(\fe_j) = - \frac{\bar z ^j}{(1+ |z|^2)}\ve _0(x) + e_j.
\]
Now we can compute $\fh _Q$.  A general holomorphic section of $Q$ can be written as $s = f^i \fe_i$, and 
\begin{eqnarray*}
\fh _Q (s_1, s_2) &=& f_1^i \overline{f_2^j} \fh _Q(\fe_i, \fe_j) =f_1^i \overline{f_2^j} \fh^* (q^{\dagger}\fe_i,q^{\dagger}\fe_j) \\
& = & f_1^i \overline{f_2^j} \fh^* \left (\frac{- \bar z^i}{1+|z|^2}(e_o +z \cdot \underbar{e}) + e_i , \frac{- \bar z^j}{1+|z|^2}(e_o +z \cdot \underbar{e}) + e_j \right ) \\
&=&f_1^i \overline{f_2^j}  \delta _{i \bar j} +  f_1^i \overline{f_2^j} \frac{z^j \bar z ^i}{1+|z|^2} - f_1^i \overline{f_2^j} \frac{z^j \bar z ^i}{1+|z|^2} - f_1^i \overline{f_2^j} \frac{z^j \bar z ^i}{1+|z|^2}\\
&=& f_1^i \overline{f_2^j} \left (  \delta _{i \bar j}  -  \frac{z^j \bar z ^i}{1+|z|^2} \right ).
\end{eqnarray*}
It is now straightforward to calculate $\Theta (\fh _Q)$ at $x_o$: with $v := b^i \fe_i \in Q_{x_o}$ and $\xi = a^i \frac{\di}{\di z^i} \in T_{\bbP(V^*), x_o}$,  
\[
\fh _Q(\Theta (\fh _Q)_{\xi \bar \xi}v , v) = |a \cdot \bar b|^2.
\]
In particular, while $Q$ is Griffiths nonnegative, a fact we also know because $Q$ is a quotient of a trivial bundle, it is not Griffiths positive.  More generally, letting $h$ be the matrix with 
\[
h_{i \bar j} = \delta _{i \bar j} - \frac{z ^j \bar z ^i}{1+|z|^2},
\]
we have 
\[
\Theta (\fh _Q)= \dbar (\di h h^{-1}) =(\dbar \di h)h^{-1}+ (\di h) h^{-1} \wedge (\dbar h) h^{-1}.
\] 
Since $h(0) = {\rm Id}$ and  $\di h(0)= \dbar h(0) = 0$, 
\[
\Theta (\fh _Q) (0) = \sum _{i,j =1} ^{r-1}  \dbar \di h_{i \bar j} (0) \tensor E^{*i} \tensor \overline{E^{*j}} = \sum _{i,j =1} ^{r-1} dz^j \wedge d\bar z ^i \tensor E^{*i} \tensor \overline{E^{*j}}.
\]
Hence for $\xi _i = a^{\ell}_i  \frac{\di}{\di z ^{\ell}} \in T^{1,0}_{\bbP(V^*),x_o}$ and $v _i = b^{\ell}_i \fe_{\ell} \in Q_{x_o}$, 
\begin{equation}\label{Q-curv-formula}
\sum _{i,j=1} ^m \fh _Q(\Theta (\fh_Q)_{\xi _i \bar \xi _j}v_i, v_j) = \sum _{i,j=1} ^m (a_i \cdot \bar b_j) \cdot \overline{(a_j \cdot \bar b _i)},
\end{equation}
which is not $m$-nonnegative for any $m > 1$.  In particular, $Q$ is not Nakano-nonnegative.

\begin{s-rmk}
Upon interchanging $\xi _i$ and $\xi _j$, we find that $Q$ is dual-Nakano nonnegative.
\red
\end{s-rmk}

If we twist $Q$ by $\cO (k)$ with its $\fh$-Fubini-Study metric $e^{-k\vp_{\fh}}$, we have 
\[
\Theta (e^{-k \vp_{\fh}} \fh _Q) = k \ii \di \dbar \vp _{\fh} \tensor {\rm Id}_Q + \Theta (\fh _Q).
\]
Using the coordinates $z$ above, we have 
\[
\ii \di \dbar \vp _{\fh} = \ii \di \dbar \log (1+|z|^2) = \frac{\ii dz \dot \wedge d\bar z }{1+|z|^2} - \frac{\ii (\bar z \cdot dz) \wedge (z \cdot d\bar z)}{(1+|z|^2)^2}.
\]
and at the origin $x_o$ we find that  $\ii \di \dbar \vp _{\fh}$ is just the Euclidean metric.  Therefore at the origin $x_o$
\[
\sum _{i,j=1} ^m e^{-k \vp_{\fh}}\fh _Q(\Theta (e^{-k \vp_{\fh}}\fh_Q)_{\xi _i \bar \xi _j}v_i, v_j) = k \sum _{i=1} ^m |a_i|^2 \sum _{j=1} ^m |b_j|^2 + \sum _{i,j=1} ^m (a_i \cdot \bar b _j) \cdot \overline{(a_j \cdot \bar b_i)}.
\]
For $m = 1$ we have 
\[
e^{-k \vp_{\fh}}\fh _Q(\Theta (e^{-k \vp_{\fh}}\fh_Q)_{\xi \bar \xi}v, v) = k|a|^2|b|^2 + |a \cdot \bar b|^2 \ge k |a|^2|b|^2,
\]
so that $\cO (k)\tensor Q$ is Griffiths positive for all $k \ge 1$.  

To understand the situation for higher positivity, we first observe that 
\begin{eqnarray*}
\sum _{i,j=1} ^m (a_i \cdot \bar b _j) \cdot \overline{(a_j \cdot \bar b_i)} &=& \sum _{i=1} ^m |a_i \cdot \bar b_i|^2 + \sum _{i < j} 2 \re (a_i \cdot \bar b _j) \cdot \overline{(a_j \cdot \bar b_i)}\\
&\ge & \sum _{i < j} 2 \re (a_i \cdot \bar b _j) \cdot \overline{(a_j \cdot \bar b_i)} \ge   - \sum _{1 \le i \neq j \le m} |a_i \cdot \bar b_j|^2.
\end{eqnarray*}
Thus 
\begin{eqnarray*}
&& k \sum _{i=1} ^m |a_i|^2 \sum _{j=1} ^m |b_j|^2 + \sum _{i,j=1} ^m (a_i \cdot \bar b _j) \cdot \overline{(a_j \cdot \bar b_i)} \\
& = & (k-1) \sum _{i=1} ^m |a_i|^2 \sum _{j=1} ^m |b_j|^2 +  \sum _{i=1} ^m |a_i|^2 \sum _{j=1} ^m |b_j|^2 + \sum _{i,j=1} ^m (a_i \cdot \bar b _j) \cdot \overline{(a_j \cdot \bar b_i)}\\
& \ge & (k-1) \sum _{i=1} ^m |a_i|^2 \sum _{j=1} ^m |b_j|^2 +  \sum _{i=1} ^m |a_i|^2 |b_i|^2 +  \sum _{1 \le i\neq j \le m} ^m \left (|a_i|^2 |b_j|^2 - |a_i \cdot b_j|^2 \right )\\
&\ge& (k-1) \sum _{i=1} ^m |a_i|^2 \sum _{j=1} ^m |b_j|^2.
\end{eqnarray*}
We have therefore proved the following proposition.

\begin{prop}\label{Q(k)-pos-prop}
The vector bundle $Q \tensor \cO (k)$ is 
\begin{enumerate}[{\rm (i)}]
\item Griffiths-nonnegative but not Griffiths-positive for $k =0$, 
\item Griffiths-positive but not $2$-positive for $k=1$, and 
\item Nakano positive as soon as $k \ge 2$.
\end{enumerate}
\end{prop}

\subsection{Non-extension Theorem for the tangent bundle of a projective plane}
Let $V$ be a complex vector space of dimension $r \ge 3$ equipped with a Hermitian inner product $\fh _o$, and denote by $\bbP(V)$ its projectivization.  As before, we let $\bbV$ denote the trivial holomorphic Hermitian vector bundle with fiber $V$, where the Hermitian metric is the constant inner product given by $\fh_o$ on each fiber.  We give $\cO (1) \to \bbP (V)$ the Fubini-Study metric $\fh _{\rm FS}$ associated to $\fh _o$, i.e., 
\[
\fh _{\rm FS} (\xi, \xi) := \frac{\left | \left < \xi, v\right >\right |^2}{\fh _o(v,v)},
\]
and we write 
\[
\bbV(1) := \bbV \tensor \cO (1).
\]
As recalled at the end of Paragraph \ref{UQB-par}, the tangent bundle of $\bbP(V)$ is characterized by the sequence of Hermitian vector bundles 
\begin{equation}\label{TP-seq-0}
0 \to \cO \to \bbV(1) \to T_{\bbP(V)} \to 0.
\end{equation}
If we twist $T_{\bbP (V)}$ by $\cO (k)$ for any positive $k$ then the resulting Hermitian vector bundle is Nakano positive by Proposition \ref{Q(k)-pos-prop}.  Thus, in view of the $L^2$ Extension Theorem \ref{OTM}, among $\cO (k) \tensor T_{\bbP(V)}$ the only possible candidate for the proof of Theorem \ref{counterexample-thm}  is the tangent bundle $T_{\bbP (V)}$.

We seek a hypersurface $S \subset \bbP (V)$ such that some global section of $K_{S} \tensor T_{\bbP(V)}|_S \to S$ does not extend to a global section of $K_{\bbP(V)} \tensor L_S \tensor T_{\bbP(V)} \to \bbP (V)$.  To study the extension problem, one looks at the \emph{adjoint restriction sequence}
\[
0 \to  \cO_{\bbP(V)}(K_{\bbP(V)} \tensor T_{\bbP(V)}) \stackrel{\tensor T}{\to} \cO_{\bbP(V)} (K_{\bbP(V)} \tensor L_S \tensor T_{\bbP(V)}) \to \cO_S (K_S \tensor E|_S) \to 0,
\]
where $T \in H^0(X, \cO (L_S))$ is any holomorphic section such that ${\rm Ord}(T) = S$.  Passing to the long exact sequence in cohomology gives
\begin{center}
\begin{tikzcd}
H^0(\bbP(V), \cO(K_{\bbP(V)} \tensor L_S \tensor T_{\bbP(V)})) \arrow[r,"R_0"]
        \arrow[d, phantom, " "{coordinate, name=Z}]
      & H^0(S, \cO_S(K_S \tensor T_{\bbP(V)}|_S)) \arrow[dl,rounded corners,to path={ -- ([xshift=2ex]\tikztostart.east)
|- (Z) [near end]\tikztonodes -| ([xshift=-2ex]\tikztotarget.west) -- (\tikztotarget)}] \\
H^1(\bbP(V), \cO _{\bbP(V)} (K_{\bbP(V)} \tensor T_{\bbP(V)})) \arrow[r,"R_1"]
& H^1(\bbP(V), \cO _{\bbP(V)} (K_{\bbP(V)} \tensor L_S \tensor T_{\bbP(V)})) 
\end{tikzcd}.
\end{center}
We want to avoid the surjectivity of the map $R_0$.  Since the sequence is exact, $R_0$ is surjective if and only if $R_1$ is injective.  

Now, letting $d := \deg (S)$, the vector bundle $L_S \tensor T_{\bbP (V)} = \cO (d) \tensor T_{\bbP(V)}$ in Nakano positive, so by Kodaira-Nakano Vanishing 
\[
H^1(\bbP(V), \cO _{\bbP(V)} (K_{\bbP(V)} \tensor L_S \tensor T_{\bbP(V)})) = \{0\}.
\]
Thus 
\[
\text{$R_1$ is injective if and only if }H^1(\bbP(V), \cO _{\bbP(V)} (K_{\bbP(V)} \tensor T_{\bbP(V)}))= \{0\}.
\]
But by Serre duality and  the Dolbeault Isomorphism 
\[
H^1(\bbP(V), \cO _{\bbP(V)} (K_{\bbP(V)} \tensor T_{\bbP(V)})) = H^{r-2} (\bbP(V), \cO_{\bbP(V)} (T_{\bbP(V)} ^*))^* \cong H^{r-2, 1}(\bbP(V)).
\]
We see that for $r \ge 4$ the restriction map $R_0$ is surjective.  On the other hand, if $r=3$ then 
\[
\dim H^1(\bbP(V), \cO _{\bbP(V)} (K_{\bbP(V)} \tensor T_{\bbP(V)})) = \dim H^{1,1}(\bbP_2) = 1,
\]
so regardless of the curve $S$, the map $R_0$ is \emph{never} surjective.    

Thus Theorem \ref{counterexample-thm} is proved if we take $X = \bbP(V)$ for any Hermitian vector space $(V, \fh)$ of dimension $3$, $E = T_{\bbP(V)}$ with metric induced by $\fh$ as above, and $S \subset \bbP(V)$ any smooth plane curve.
\qed

\section{Proof of Theorem \ref{dit-triv}}\label{dit-appendix}

In the proof of Theorem \ref{dit-triv} we shall need a vector bundle analogue of the part of Schur Complement Theory originally used by Semmes in \cite{semmes} to study the curvature of $\sH \to B$ in the case in which $E$ has rank $1$.  We begin with two vector spaces $V$ and $M$; the choice of letters is meant to suggest that $V$ is a fiber of the vector bundle and $M$ is the tangent space to the manifold.  A Hermitian form on $V \tensor M$ is said to be $k$-positive if it is positive on tensors of rank $k$.  This notion was previously introduced for Hermitian endomorphisms (c.f. Definition \ref{pos-endos-defn}), but the idea is the same for Hermitian forms.

In our situation the manifold is a product, so we suppose that $M = M_1 \oplus M_2$.  We fix a Hermitian inner product $g$ for $M$ such that $M_1 \perp M_2$, and a Hermitian inner product $h$ for $V$.  Let $\fH$ be a Hermitian form on $M \tensor V$.  There is a unique linear $g \tensor h$-Hermitian map $\fJ : M \times V \to M \times V$, sometimes called the \emph{inertia map of $\fH$} with respect to $g \tensor h$, such that 
\[
\fH (T_1, T_2) = (g \tensor h)(\fJ T_1, T_2).
\]
The splitting $M = M_1 \oplus M_2$ induces a decomposition of $\fJ$ into four maps 
\[
\fJ _{ij} :M_i \tensor V \to M_j \tensor V, \quad 1 \le i,j\le 2,
\]
and since $\fJ$ is Hermitian, $ \fJ^{\dagger}_{ij} = \fJ _{ji}$ for all $1 \le i,j\le 2$,
where the Hermitian conjugation is with respect to the appropriate Hermitian inner products.

\begin{prop}\label{positivity-of-dets}
Let $\fH$ be a $k$-positive Hermitian form on $(M_1 \oplus M_2) \tensor V$ such that $\fH|_{M_2 \tensor V}$ is Nakano-positive.  Then the Hermitian form on $M_1 \tensor V$ whose inertia map with respect to $g_1 \tensor h$ is
\[
\fI _{11} - \fI_{12} \fI_{22}^{-1} \fI _{21} : M_1 \tensor V \to M_1 \tensor V
\]
is $k$-positive.
\end{prop}

\begin{proof}
For $v\in V$, and define $v^{\dagger} \in V^*$ by 
\[
v^\dagger w = h(w,v) \quad \text{for all }w \in V,
\]
and for $\xi \in M$ let $\xi^{\ddagger} \in M^*$ be defined by 
\[
\xi ^\ddagger \eta = g(\eta,\xi) \quad \text{ for all }\eta \in M.
\]
For $f_i \in V$ and $\xi_k= (\tau_k, w_k) \in M = M_1 \oplus M_2$, $1 \le i,k \le m$,  we have 
\begin{eqnarray*}
\fH (\xi _k \tensor f_i, \xi _{\ell} \tensor f_j) &=& \tau _{\ell} ^{\ddagger} \left (f_j ^{\dagger} \fI _{11} f_i\right ) \tau _k  +  w_{\ell} ^{\ddagger} \left (f_j ^{\dagger} \fI _{12} f_i\right ) \tau _k +  \tau_{\ell} ^{\ddagger} \left (f_j ^{\dagger} \fI _{21} f_i\right ) w _k +  w_{\ell} ^{\ddagger} \left (f_j ^{\dagger} \fI _{22} f_i\right ) w_k \\
&=& \tau _{\ell} ^{\ddagger} \left (f_j ^{\dagger} \fI _{11} f_i\right ) \tau _k  +  w_{\ell} ^{\ddagger} \left (f_j ^{\dagger} \fI _{12} f_i\right ) \tau _k +  \tau_{\ell} ^{\ddagger} \left (f_j ^{\dagger} \fI _{12}^{\dagger} f_i\right ) w _k +  w_{\ell} ^{\ddagger} \left (f_j ^{\dagger} \fI _{22} f_i\right ) w_k;
\end{eqnarray*}
the second equality holds because $\fI$ is Hermitian.  Now take $w_k := - \fJ_{22} ^{-1} \fJ_{21} \tau _k$, $1 \le k \le m$.  Then 
\[
\fH (\xi _k \tensor f_i,\xi _{\ell} \tensor f_j) = f_j^{\dagger} \tensor \tau _{\ell} ^{\ddagger} \left ( \fJ _{11} - \fJ_{12} \fJ_{22} ^{-1} \fJ_{21} \right ) f_i \tensor \tau _k.
\]
Setting $k=i$ and $\ell = j$, and then summing, we obtain 
\[
\sum _{i,j=1} ^m \fH (\xi _i \tensor f_i,\xi _{j} \tensor f_j)= \sum _{i,j=1} ^m f_j^{\dagger} \tensor \tau _{j} ^{\ddagger} \left ( \fJ _{11} - \fJ_{12} \fJ_{22} ^{-1} \fJ_{21} \right ) f_i \tensor \tau _i,
\]
and the result follows.
\end{proof}

Before turning to the proof of Theorem \ref{dit-triv}, we introduce the notation 
\[
\sL _t := C\ell \left \{ f \in \Gamma (X, \sC^{\infty}(K_X \tensor E))\ ;\ \int _X |f|^2_{\fh ^{(t)}} < +\infty \right \}
\]
for the Hilbert space of not-necessarily holomorphic, square integrable $E$-valued $(n,0)$-forms on $X$, with respect to the $L^2$ norm associated to the metric $\fh^{(t)} := \fh (\cdot, t)$.  (Here $C\ell$ means the Hilbert space closure.)  We note that $\sH _t \subset \sL_t$ is a closed subspace, and thus there are orthogonal projections 
\[
P_t : \sL _t \to \sH_t \quad \text{and} \quad N_t = {\rm Id} - P_t: \sL _t \to \sH_t ^{\perp},
\]
often called the Bergman projection and the Neumann operator respectively.  

\begin{s-rmk}
As Berndtsson notes in \cite{bo-annals}, $\sL_t$ are the fibers of a holomorphic Hilbert bundle $\sL \to B$ that contains $\sH \to B$ as a subbundle.  Though we don't use the Gauss-Griffiths curvature formula, instead computing the curvature of $\sH$ directly, it is helpful to be aware of this picture.   
\red
\end{s-rmk}

\begin{proof}[Proof of Theorem \ref{dit-triv}]  The proof proceeds in two steps.
\

\medskip 

\noi {\sc Step 1:  the case in which the metrics $\fh^{(t)}$, $t \in B$, are Nakano-positive}.  
We use Proposition \ref{k-prop}.  We may assume that $B$ is a coordinate neighborhood of $o$ in $\bbC^m$ with coordinates $t^1,...,t^m$ and that $f_1,...,f_k \in \Gamma _{\cO} ( B, \sH(\fh))$ satisfy $(D^{1,0} f_j )(o)= 0$.

Let $A(\fh)$ denote the connection form of $\fh$.  Then
\[
D^{1,0}_{\frac{\di}{\di t ^j}} f =  P_{t} \left ( \frac{\di f }{\di t ^j} + \frac{\di}{\di t ^j} \lrcorner A(\fh)  f \right ).
\]
Writing $\frac{\di}{\di t ^j} \lrcorner A(\fh)  = A_j (\fh)$, we have
\begin{eqnarray*}
&& - \ii \di _{B}\dbar _{B} \left ( \int _{X_o}\left < f_{\lambda} , f_{\mu}\right >_{\fh^{(t)}}\right ) \wedge \Upsilon ^{\lambda \bar \mu} = \ii \dbar _{B} \di _{B}  \left ( \int _{X_o} \left < f_{\lambda} , f_{\mu}\right >_{\fh^{(t)}} \right ) \wedge \Upsilon ^{\lambda \bar \mu} \\
&=& \ii \dbar _{B} \left (\int _{X_o} \left <\frac{\di f_{\lambda}}{\di t^{\sigma}} + A_{\sigma} (\fh) f_{\lambda}, f_{\mu} \right >_{\fh^{(t)}}\right ) \wedge d t^{\sigma} \wedge \Upsilon^{\lambda \bar \mu}\\
&=& \left (\int _{X_o} \left < \left (\tfrac{\di}{\di \bar t ^{\nu}} A_{\sigma} (\fh) \right )f_{\lambda} , f_{\mu}\right >_{\fh^{(t)}} \right ) \ii d\bar t ^{\nu} \wedge dt ^{\sigma} \wedge \Upsilon^{\lambda \bar \mu}\\
&& \qquad  +\left (\int _{X_o} \left < \frac{\di f_{\lambda}}{\di t  ^{\sigma}} + A_{\sigma} (\fh) f_{\lambda},  \frac{\di f_{\mu} }{\di t ^{\nu}} + A_{\nu}(\fh) f_{\mu}\right >_{\fh^{(t)}}\right ) \ii d\bar t ^{\nu} \wedge dt^{\sigma} \wedge \Upsilon^{\lambda \bar \mu}\\
&=&   \left (\int _{X_o} \delta ^{\sigma \lambda} \delta ^{\bar \nu \bar \mu} \left < \left (-\tfrac{\di}{\di t ^{\sigma}} A_{\bar \nu} (\fh) \right ) f_{\lambda} , f_{\mu} \right >_{\fh^{(t)}} - \int _{X_o} \left |\delta^{\sigma \lambda} \left ( \frac{\di f_{\lambda}}{\di t  ^{\sigma}} + A_{\sigma} (h) f_{\lambda}\right ) \right |^2_{\fh^{(t)}} \right ) dV(t).
\end{eqnarray*}
By Pythagoras' Theorem 
\begin{eqnarray*}
&& \int _{X_o}\left  |\delta^{\sigma \lambda} \left ( \frac{\di f_{\lambda}}{\di t  ^{\sigma}} + A_{\sigma} (\fh) f_{\lambda}\right ) \right |^2_{\fh^{(t)}} \\
&=&  \int _{X_o}\left  |P_{t } \left ( \delta^{\sigma \lambda} \left ( \frac{\di f_{\lambda}}{\di t  ^{\sigma}} + A_{\sigma} (\fh) f_{\lambda}\right ) \right )\right |^2_{\fh^{(t)}}  +  \int _{X_o}\left  |N_{t } \left ( \delta^{\sigma \lambda} \left ( \frac{\di f_{\lambda}}{\di t  ^{\sigma}} + A_{\sigma} (\fh) f_{\lambda}\right ) \right )\right |^2_{\fh^{(t)}} \\
&=&  \int _{X_o}\left  |\delta ^{\sigma \lambda} D^{1,0} _{\frac{\di}{\di t  ^{\sigma}}} f_{\lambda}\right |^2_{\fh^{(t)}}  +  \int _{X_o}\left  |N_{t } \left ( \delta^{\sigma \lambda} A_{\sigma} (\fh) f_{\lambda}\right ) \right |^2_{\fh^{(t)}}.
\end{eqnarray*}
(Note that $N_t(\frac{\di f_{\lambda}}{\di t  ^{\sigma}}) = 0$ because $\frac{\di f_{\lambda}}{\di t  ^{\sigma}}$ is holomorphic.) 
Since $D^{1,0} _{\frac{\di}{\di t  ^{\sigma}}} f_{\lambda}$ vanishes at $t=o$, Theorem \ref{hormander-estimate} yields the estimate
\begin{eqnarray*}
&& \frac{1}{dV(o)} \left ( - \ii \di _{B}\dbar _{B} \left ( \int _{X_o}\left < f_{\lambda} , f_{\mu} \right >_{\fh^{(t)}} \right ) \wedge \Upsilon ^{\lambda \bar \mu} \right )_{t =o} \\
&& \ge \delta ^{\sigma \lambda} \delta ^{\bar \nu \bar \mu}  \int _{X_o} \!\!\!\! \left <\left (-\tfrac{\di}{\di t ^{\sigma}} A_{\bar \nu} (\fh) \right ) \! f_{\lambda} ,f_{\mu} \right >_{\fh^{(o)}}  \! - \! \delta ^{\sigma \lambda} \delta ^{\bar \nu \bar \mu}  \int _{X_o}\!\!\!\!  \left < (\Theta (\fh^{(o)})^{\sharp_g}) ^{-1} \dbar _{X_o} \!  A_{\sigma} (\fh) f_{\lambda}, \dbar _{X_o} A_{\nu} (\fh)\! f_{\mu} \right >_{\fh^{(o)}}\\
&&  =\int _{X_o} \!\!\!\! \delta ^{\sigma \lambda} \delta ^{\bar \nu \bar \mu} \left <\left [ \left (-\tfrac{\di}{\di t ^{\sigma}} A_{\bar \nu} (\fh) \right ) - \left (  \dbar _{X_o} A_{\nu} (\fh) \right ) ^{\dagger} (\Theta (\fh^{(o)})^{\sharp _g})^{-1} \dbar _{X_o}  A_{\sigma} (\fh)\right ]  f_{\lambda}, f_{\mu} \right >_{\fh^{(o)}} .
\end{eqnarray*}
But by Proposition \ref{positivity-of-dets} the integrand in the last line is non-negative (resp. positive) when $\fh$ is $k$-nonnegative (resp. $k$-positive).  By Proposition \ref{k-prop} Theorem \ref{dit-triv} is proved in the positive case.

\medskip 

\noi {\sc Step 2: The general, nonnegative case.}
As the ambient manifold $Y_o$ is Stein, there exists a strictly plurisubharmonic exhaustion $\rho : Y_o \to \bbR$ so that the metrics $e^{-\ve \rho} \fh ^{(t)}$ are strictly Nakano-positive on the fibers of $X_o \times B \to B$.  Let us write $\fh _{\ve} := e^{- \ve p _1 ^* \rho} \fh$, where $p _1 : X_o \times B \to X_o$ is the projection.   As in the case $\ve = 0$, 
\begin{eqnarray*}
&& - \frac{1}{dV(t)}\ii \di _{B}\dbar _{B} \left ( \int _{X_o}\left < f^{\ve}_{\lambda} , f^{\ve}_{\mu}\right >_{\fh^{(t)}_{\ve}}\right ) \wedge \Upsilon ^{\lambda \bar \mu} \\
&=& \int _{X_o} \delta ^{\sigma \lambda} \delta ^{\bar \nu \bar \mu} \left < \left (-\tfrac{\di}{\di t ^{\sigma}} A_{\bar \nu} (\fh_{\ve} ) \right ) f^{\ve}_{\lambda} , f^{\ve}_{\mu} \right >_{\fh^{(t)}_{\ve}} -\int _{X_o}\left  |\delta ^{\sigma \lambda} D^{\ve} _{\frac{\di}{\di t  ^{\sigma}}} f^{\ve}_{\lambda}\right |^2_{\fh^{(t)}_{\ve}}  -  \int _{X_o}\left  |N^{\ve}_{t } \left ( \delta^{\sigma \lambda} A_{\sigma} (\fh_{\ve}) f^{\ve}_{\lambda}\right ) \right |^2_{\fh^{(t)}_{\ve}}.
\end{eqnarray*}
Here $f^{\ve}_{\lambda} \in \cO (E)_o$ depend smoothly on $\ve$ and have $\fh_{\ve}$-covariant derivative $D^{\ve}$ that vanishes at $t= o$.  As above, the right hand side is (semi-)positive at $o$ if $\fh _{\ve}$ is $m$-(semi-)positive.  It therefore suffices to show that 
\begin{equation}\label{cty-of-projection}
\lim _{\ve \to 0}  \int _{X_o} \left |  N_t ^{\ve} f \right |^2 _{\fh _{\ve} ^{(t)}} = \int _{X_o} \left |  N_t f \right |^2 _{\fh ^{(t)}}  \quad \text{for any $f \in L ^2(\fh _{\ve} ^{(t)}) = L ^2(\fh ^{(t)})$}.
\end{equation}
To establish \eqref{cty-of-projection}, note that $N^{\ve}_t f$ is the nearest section $F_{\ve}$ to $f$, in the $L ^2(\fh _{\ve} ^{(t)})$-distance, that satisfies 
\[
\int _{X_o} \fh _{\ve}^{(t)}( F_{\ve} , g ) = 0 \quad \text{for all } g\in \sH ^2(\fh _{\ve})_t = \sH ^2(\fh)_t.
\]
Therefore 
\[
\int _{X_o} e^{- \ve \rho} \fh ^{(t)} (N^{\ve}_t f, g ) = 0 \quad \text{for all } g\in \sH ^2(\fh)_t.
\]
On the other hand, $N_t f$ is the nearest section $F$ to $f$, in the $L ^2(\fh^{(t)})$-distance, that satisfies
\[
\int _{X_o} \fh (F,  g) = 0 \quad \text{for all } g\in \sH^2(\fh)_t.
\]
Therefore 
\[
||N_tf - f||\le ||e^{- \ve \rho}N_t^{\ve} f -f|| \le || e^{- \ve \rho}(N_t^{\ve} f -f )|| + ||(1-e^{- \ve \rho})f|| \le ||N_t^{\ve}f - f|| + (e^{\ve M}-1)||f||,
\]
where $M = \sup _{X_o} \rho$ and $|| \cdot ||$ denotes the norm in $L ^2(\fh^{(t)})$.  Similarly 
\[
||N_t^{\ve}f - f||_{\ve} \le e^{\ve M} ||N_tf- f||_{\ve} + (e^{\ve M} -1)||f||_{\ve},
\]
where $|| \cdot ||_{\ve}$ denotes the norm in $L ^2(\fh_{\ve}^{(t)})$.  One also has the estimates 
\[
||N_t ^{\ve}f - f||_{\ve} \ge e^{-\ve M} ||N_t ^{\ve}f - f|| \quad \text{and} \quad ||f||_{\ve} \le ||f||.
\]
Taking limits shows that $\displaystyle{\lim_{\ve \to 0} ||N_t^{\ve} f - f|| = ||N_tf-f||}$.  Since $N_t^{\ve}f = e^{-\ve \rho} N_t^{\ve} f + (1-e^{-\ve\rho})N_t^{\ve}f$ and $||N_t^{\ve}f ||^2 \le e^{\ve M} ||N_t^{\ve} f||^2_{\ve} \le e^{\ve M} ||f||^2_{\ve} \le e^{\ve M} ||f||^2$, $N_t^{\ve} f$ converges to an element of $\sH_t^{\perp}$, so we must have $\displaystyle{\lim_{\ve \to 0} N_t^{\ve}f = N_tf}$, and \eqref{cty-of-projection} follows by Lebesgue's Dominated Convergence Theorem.  The proof of Theorem \ref{dit-triv} is complete. 
\end{proof}

\section{Proofs of Theorem \ref{OTM}}\label{OTM-proofs-section}

\subsection{Proof as a corollary of \cite[Theorem 2.1]{gz}}

We refer to \cite{gz} for the notation and the following result (except that our complex manifold is $X$ and our hypersurface is $Z$).

\begin{thm}[\text{\cite[Theorem 2.1 in the case of hypersurfaces]{gz}}]\label{gz-thm}
Let $X$ be a Stein manifold, let $Z \subset X$ be a closed complex hypersurface and let $E \to X$ be a holomorphic vector bundle with smooth Hermitian metric $\fh$.  Fix a function $\Psi \in \# _A(Z) \cap \sC^{\infty}(X\setminus Z)$ satisfying 
\begin{enumerate}[{\rm (i)}]
\item $e^{-\Psi} \fh$ is semipositive in the sense of Nakano on $X\setminus Z$, and 
\item there exists a continuous function $a: (-A, +\infty] \to \bbR$ such that $0 < a(t) \le s(t)$ and 
\[
a(-\Psi) \ii \Theta (e^{-\Psi}\fh) + \ii \di \dbar \Psi \ge _{\rm Nak} 0 \quad \text{on }X \setminus Z,
\]
where 
\[
s(t) = \frac{\int _{-A} ^t \left ( \frac{1}{\delta}c_A(-A)e^A + \int _{-a} ^{t_2} c_A(t_1) e^{-t_1}dt_1 \right ) dt_2 + \frac{1}{\delta ^2}c_A(-A)e^A}{\frac{1}{\delta} c_A(-A)e^A + \inf _{-A} ^t c_A(t_1)e^{-t_a} dt_1}.
\]
\end{enumerate} 
Then for any $f \in H^0(Z, \cO _Z((K_X \tensor E)|_Z))$ satisfying 
\[
\pi \int _Z |f|^2_{\fh,dV_M} dV_M[\Psi] < +\infty
\]
there exists $F \in H^0(X, \cO _X(K_X \tensor E))$ such that 
\[
F|_Z = f \quad \text{and} \quad \int _X c_A(-\Psi) |F|^2_{\fh} \le \left (\frac{e^A}{\delta} c_A(-A) + \int _{-A} ^{\infty} c_A(t) e^{-t} dt \right ) \pi \int _Z |f|^2_{\fh,dV_M} dV_M[\Psi].
\]
Here $\delta > 0$ is any given number, and $c_A\in \sC^{\infty}((-A, +\infty))$ is any positive function such that 
\begin{enumerate}[{\rm (a)}]
\item $\displaystyle{\int _{-A} ^{+\infty} c_A(t) e^{-t} dt < +\infty}$,  and 
\item 
\begin{align*}
& \left (\frac{e^A}{\delta} c_A(-A) + \int _{-A} ^t c_A(t_1)e^{-t_1} dt_1 \right )^2 \\
& \qquad > c_A(t)e^{-t} \left ( \int _{-A} ^t \left ( \frac{e^A}{\delta} c_A(-A) + \int _{-A} ^{t_2}c_A(t_1)e^{-t_1}dt_1\right ) dt_2 + \frac{e^A}{\delta ^2}c_A(-A) \right ).
\end{align*}
\end{enumerate}
\end{thm}

The function $\Psi$ in our case is just 
\[
\Psi := \log |T|^2_{\fg}.
\]
We claim that Ohsawa's measure $dV_M[\Psi]$ is just 
\[
dV_M[\Psi] = \frac{dV_M}{dV_g}\frac{dA_g}{|dT|^2_{\fg, g}}.
\]
To prove the claim, we need to show that any nonnegative continuous function $\chi \in \sC_o(X)$ we have 
\[
\limsup _{t \to \infty} \frac{2}{\pi} \int _{X\setminus Z} \chi e^{-\Psi}\mathbf{1}_{\{-1-t< \Psi < -t\}} dV_g = \int _Z \frac{\chi dA_g}{|dT|^2_{\fg,g}}.
\]
But for $t >> 0$ we have 
\begin{eqnarray*}
\frac{2}{\pi} \int _{X\setminus Z} \chi e^{-\Psi}\mathbf{1}_{\{-1-t< \Psi < -t\}} dV_g
&=& \frac{1}{2 \pi}  \int _{\{ e^{-t-1} \le |T|^2_{\fg} \le e^{-t}\}} \frac{\chi}{|T|^2_{\fg}} \frac{\left < \ii \nabla T \wedge \nabla \bar T, \fg \right > }{|\nabla T|^2_{\fg,d}} \wedge \frac{ \omega _g ^{n-1}}{(n-1)!}\\
&=& \int _{e^{-t-1}}^{e^{-t}} \int _Z\left ( \frac{\chi}{|\nabla T|^2_{\fg,g}} \frac{\omega_g^{n-1}}{(n-1)!} \right ) \frac{dr}{r} \\
&= & \left (\int _{e^{-t-1}}^{e^{-t}} \frac{dr}{r} \right ) \left ( \int _Z  \chi \frac{dA_g}{|dT|^2_{\fg, g}}  + o (1) \right ) = \int _Z  \chi \frac{dA_g}{|dT|^2_{\fg, g}} + o(1).
\end{eqnarray*}
(This calculation is often called the \emph{Co-area Formula}.)

Away from $Z$ one has 
\[
\ii \di \dbar \Psi = - \ii\Theta (\fg) \quad \text{and} \quad \ii \Theta (e^{-\Psi}\fh) = \ii \Theta (\fh) - \ii \Theta (\fg).
\]
Thus to achieve our curvature conditions we must take 
\[
a(t) = \frac{1}{\delta}.
\]
We take $A=0$ and $c_A (t) \equiv 1$.  It is easy, though perhaps a bit tedious, to check that all of the numerous requirements are satisfied by these choices.  We therefore see that for each $f$ we get an extension $F$ satisfying 
\[
\int _X |F|^2_{\fh, g} dV_g \le \left ( \frac{1}{\delta} + 1\right ) \pi \int _Z \frac{|f|^2_{\fh,g} dA_g}{|dT|^2_{\fg,g}}.
\]
After an application of the Adjunction Formula, Theorem \ref{OTM} follows.

\subsection{A second proof, under stronger curvature hypotheses}

In this paragraph we give a proof of Theorem \ref{OTM} under the stronger curvature hypothesis 
\begin{equation}\label{curv-hyp-OTBL-vb-c}
\ii \Theta (\fh) \ge_{{\rm Nak}} \delta \ii \Theta (\fg) \tensor {\rm Id}_E\ge_{{\rm Nak}} 0.
\end{equation}
This hypothesis implies the curvature condition \eqref{curv-hyp-OT}, but can be strictly stronger.  The proof follows the ideas of Berndtsson and Lempert \cite{bl}.

\subsubsection{\sc The non-canonical formulation of Theorem \ref{OTM}}

The canonical (or \emph{adjoint}) formulation of $L^2$ extension, i.e., the question of extending sections of $K_Z \tensor E|_Z$ to sections of $K_X \tensor L_Z \tensor E$, in which Theorem \ref{OTM} is stated, is useful in complex geometry in part because it allows one to define $L^2$ norms without selecting a volume form.  However, this feature ceases to be convenient when one wants to change coordinates `almost everywhere', because the canonical bundle does not remain invariant under such transformations.  In such and perhaps other situations, the following equivalent formulation can be convenient.

\begin{thm}\label{OTMnc}
Let $X$ be an essentially Stein manifold equipped with a K\"ahler metric $g$, let $Z$ be a smooth hypersurface cut out by a holomorphic section $T \in H^0(X, \cO (L_Z))$, and let $\fg$ be a smooth metric for $L_Z \to X$ such that 
\[
\sup _X |T|^2_{\fg} \le 1.
\]
Let $E \to X$ be a holomorphic vector bundle with smooth Hermitian metric $\fh$ such that 
\begin{equation}\label{curv-hyp-OTnc}
\ii \Theta (\fh) + {\rm Ricci}(g) \ge_{{\rm Nak}} (1+t \delta) \ii \Theta (\fg) \tensor {\rm Id}_E \quad \text{for all }t \in [0,1].
\end{equation}
Then for every $f \in H^0(Z, \cO_Z (E|_{Z}))$ satisfying
\[
\int _Z \frac{|f|^2_{\fh}dA_{g}}{|dT|^2_{\fg,g}} < +\infty
\]
there exists $F \in H^0(X, \cO _X(E))$ such that 
\[
F|_Z = f \quad \text{and} \quad \int _X |F|^2_{\fh} dV_{g}\le \frac{\pi (1+\delta)}{\delta} \int _Z \frac{|f|^2_{\fh} dA_{g}}{|dT|^2_{\fg,g}}.
\]
\end{thm}

\begin{s-rmk}
The equivalence of Theorems \ref{OTMnc} and \ref{OTM} is obtained as follows.  Let $\cE := K_X^* \tensor L_Z^* \tensor E$, equipped with the metric $\mathbf{h} := dV_{\omega}\tensor \fg^* \tensor \fh$.  Then 
\[
E = K_X \tensor L_Z \tensor \cE, \quad E|_Z= K_Z \tensor \cE|_Z \quad \text{and} \quad  \Theta (\mathbf{h}) = \Theta (\fh) + {\rm Ricci}(\omega)  - \Theta (\fg);
\]
the second equality is of course the adjunction formula which, at the level of sections, identifies a section $f$ of $E|_Z = (K_X \tensor L_Z \tensor \cE)|_Z$ with the section $f/dT$ of $K_Z \tensor \cE|_Z$.  Thus 
\[
\ii \Theta (\mathbf{h}) - t\delta \ii \Theta (\fg) = \ii \Theta (\fh) + {\rm Ricci}(\omega) - (1+t\delta)\ii \Theta (\fg) \quad \text{and} \quad |f|^2_{\mathbf{h}} = \frac{|f|^2 _{\fh}dA_{\omega}}{|dT|^2_{\fg, \omega}},
\]
where the second equality is by the co-area formula.  In particular, the metrics $\fh$ and $g$ satisfy \eqref{curv-hyp-OTnc} if and only if the metric $\mathbf{h}$ satisfies the curvature condition \eqref{curv-hyp-OT} of Theorem \ref{OTM}.

Suppose then that Theorem \ref{OTM} holds.  Given $f \in H^0(Z, \cO (E|_Z)) = H^0(Z, \cO (K_Z \tensor \cE|_Z))$, Theorem \ref{OTM} provides $F \in H^0(X, \cO (K_X \tensor L_Z \tensor \cE))=H^0(X, \cO (E))$ such that $F|_Z = f$ and 
\[
\int _X |F|^2_{\fh} dV_{\omega} = \int _X |F|^2_{\fg \tensor \mathbf{h}} \le \frac{\pi (1+\delta)}{\delta} \int _Z |f|^2 _{\mathbf{h}} = \frac{\pi (1+\delta)}{\delta} \int _Z \frac{|f|^2_{\fh} dA_{\omega}}{|dT|^2_{\fg,\omega}},
\]
so Theorem \ref{OTMnc} holds.  The steps can be reversed to deduce Theorem \ref{OTM} from Theorem \ref{OTMnc}.
\red
\end{s-rmk}

Under the equivalence just outlined, the curvature condition \eqref{curv-hyp-OTBL-vb-c} is replaced by 
\begin{equation}\label{curv-hyp-OTBL-vb}
\ii \Theta (\fh) + {\rm Ricci}(\omega) - \ii \Theta (\fg) \tensor {\rm Id}_E \ge_{{\rm Nak}} \delta \ii \Theta (\fg) \tensor {\rm Id}_E\ge_{{\rm Nak}} 0.
\end{equation}
The proof then unfolds in two steps.  
\begin{enumerate}[i.]
\item In the first step we assume that the hypersurface $Z$ is cut out by a bounded holomorphic function; we call this the \emph{flat case}, since the normal bundle $L_Z \to X$ is trivial, and we can take the metric $\fg$ to be flat.
\item In the second step we reduce to the flat case using an idea of Grauert.
\end{enumerate}

\subsubsection{\sc The Flat Case}

We begin by proving the following special case of Theorem \ref{OTMnc}. 

\begin{thm}\label{ot-vb-flat}
Let $X$ be a Stein manifold with K\"ahler metric $g$ and let $T : X \to \bbD$ be a holomorphic function such that $dT$ is not identically zero on any irreducible component of $Z := T^{-1}(0)$.  Let $E \to X$ be a holomorphic vector bundle with smooth Hermitian metric $\fh$ such that 
\[
\ii \Theta (\fh) + {\rm Ricci}(\omega) \ge_{{\rm Nak}} 0.
\]
Then for any $f \in H^0 (Z ,\cO(E))$ such that 
\[
\int _Z \fh (f,f)\frac{dA_{g}}{|dT|^2_{g}} < +\infty
\]
there exists $F \in H^0(X, \cO(E))$ such that 
\[
F|_Z = f \quad \text{and} \quad \int _X \fh(F, F)dV_{g} \le \pi \int _Z \fh (f,f)\frac{dA_{g}}{|dT|^2_{g}}.
\]
\end{thm}

\noi (Recall that integration on an analytic variety means integration over the regular locus.)

By using results from real analysis (namely, the Banach-Alaoglu Theorem and the Lebesgue limit theorems), it is standard that the proof reduces to the following situation.

\begin{enumerate}[i.]
\item  The manifold $X$ is a relatively compact, pseudoconvex domain in a Stein manifold $\tilde X$, and there is a smooth complex hypersurface $\tilde Z \subset \tilde X$ such that $\tilde Z \cap X = Z$.
\item  There is a holomorphic vector bundle $\tilde E \to \tilde X$ and a smooth Hermitian metric $\tilde \fh$ for $\tilde E$ such that $\tilde E |_X = E$ and $\tilde \fh |_X = \fh$.
\item  There is a holomorphic section $\tilde f \in H^0(\tilde Z, \cO (\tilde E))$ such that $\tilde f|_Z = f$.
\end{enumerate}
Perhaps the least standard part is the smoothness of $\tilde Z$; this situation can be achieved because the singular locus of a hypersurface $Z_o$ in a Stein manifold is contained in some hypersurface $W$ such that no component of $Z_o$ lies entirely in $W$.  If we remove $W$ from our original Stein manifold then the resulting manifold is still Stein.  And since at the end of the day we are integrating, the removal of a null set has no effect.  

By Stein theory the section $\tilde f$ extends to a holomorphic section $\tilde F \in H^0(\tilde X, \cO (\tilde E))$, albeit with no estimates.  Nevertheless, its restriction $F:= \tilde F|_X$ to $X$ has finite $L^2$ norm because of the smoothness of $\fh$ and the boundedness of $X$.  It follows that there exists an extension of minimal $L^2$ norm, and our task is to estimate this minimal norm.  From here on fix $f \in H^0(\overline{Z}, \cO (E))$ such that 
\[
\int _Z \fh (f,f) \frac{dA_g}{|dT|^2_g} < +\infty,
\]
and denote by $F_o$ the extension of $f$ whose squared norm 
\begin{equation}\label{x-norm}
\int _X\fh (F_o,F_o) dV_g
\end{equation}
is minimal.  We denote by $\sH (\fh, g)$ the Hilbert space of all $F \in H^0(X, \cO (E))$ whose squared norm \eqref{x-norm} is finite.

One characterizes the norm of the minimal extension via duality as follows.  Let 
\[
\fI _{\fh}(Z) = \left \{ G \in \sH (\fh, g)\ ;\ G|_Z \equiv 0\right \}
\]
and 
\[
{\rm Ann}(\fI _{\fh}(Z)) := \left \{ \xi \in \sH (\fh, g)^* \ ;\ \left < \xi, G \right > = 0 \text{ for all }G \in \fI_{\fh}(Z) \right \}.
\]
Then 
\begin{equation}\label{dual-norm}
\int _X\fh (F_o,F_o) dV_g = \sup \left \{ \frac{\left | \left < \xi, F\right > \right |^2}{||\xi||^2_*} \ ;\ \xi \in {\rm Ann}(\fI_{\fh}(Z))\right \},
\end{equation}
where $F \in \sH (\fh, g)$ is any extension of $f$.  (The right hand side of \eqref{dual-norm} does not depend on $F$.)

The approach to proving Theorem \ref{ot-vb-flat} is to estimate the right hand side of \eqref{dual-norm}.  To obtain such an estimate, one begins by observing that, with 
\[
\xi _{\sigma} : \sH (\fh , g) \ni F \mapsto \int _Z \fh (F, \sigma) \frac{dA_g}{|dT|^2_g},
\]
the set 
\[
\left \{ \xi _{\sigma}\ ;\ \sigma \in H^0(Z, \sC^{\infty}(E)) \cap \sL ^2_Z (\fh, g) \text{ compactly supported}\right \}
\]
is dense in ${\rm Ann}(\fI _{\fh}(Z))$, where 
\[
\sL ^2_Z (\fh, g) := \left \{ \sigma \in \Gamma (Z, E) \text{ measurable} \ ;\ \int _Z \fh (\sigma, \sigma ) \frac{dA_g}{|dT|^2_g} < +\infty \right \}.
\]
The subspace 
\[
\sI (\fh , g) := \sL ^2 _Z (\fh , g) \cap H^0(Z, \cO(E))
\]
of holomorphic sections in $\sL ^2_Z(\fh, g)$ is closed, and thus we have an orthogonal projection 
\[
P : \sL ^2_Z (\fh, g) \to \sI (\fh , g).
\]
In terms of this orthogonal projection, we find that for any extension $F$ of $f$, 
\begin{eqnarray*}
\left | \left < \xi _{\sigma} , F \right > \right |^2 &=&  \int _Z \fh (f, {\sigma}) \frac{dA_g}{|dT|^2_g} \\
&=&  \int _Z \fh (f,P {\sigma}) \frac{dA_g}{|dT|^2_g} \le \left ( \int _Z \fh (f,f) \frac{dA_g}{|dT|^2_g} \right ) \left ( \int _Z \fh (P{\sigma}, P{\sigma}) \frac{dA_g}{|dT|^2_g} \right ).
\end{eqnarray*}
Thus to prove Theorem \ref{ot-vb-flat} it suffices to show that 
\begin{equation}\label{key-estimate-ot-vb-flat}
\int _Z \fh (P{\sigma}, P{\sigma}) \frac{dA_g}{|dT|^2_g} \le \pi ||\xi _{\sigma}||^2_*
\end{equation}
The strategy for proving \eqref{key-estimate-ot-vb-flat} is essentially to degenerate $X$ to $Z$, in a sense that we now explain.

Let 
\[
\bbL := \{ \tau \in \bbC\ ;\ \Re \tau < 0 \}.
\]
denote the left half plane.  For each $\tau \in \bbL$ define 
\[
\sX := \left \{ (x, \tau) \in X \times \bbL \ ;\ \log |T(x)|^2 < \Re \tau \right \} \quad \text{and} \quad X_{\tau} := \{ x \in X\ ;\ (x,\tau) \in \sX \}.
\]
Ideally one would like to work with the Hilbert field $\mathbf{H}(\fh, g) \to \bbL$ whose fiber over $\tau \in \bbL$ is 
\[
\mathbf{H}_{\tau}(\fh , g) := \left \{ F \in H^0(X_{\tau}, \cO (E))\ ;\ e^{-\Re \tau} \int _{X_{\tau}} \fh (F,F) dV_g < +\infty \right \}.
\]
The degeneration is as $t := \re \tau \to -\infty$.  The following lemma, which is an easy consequence of Fubini's Theorem, will be of use later.  

\begin{lem}\label{fubini}
Let $\cF$ be a smooth function on $\overline{X}$.  Then 
\[
\limsup _{t \to -\infty} e^{-t} \int _{X_t} \cF dV_g= \pi \int _Z \cF\frac{dA_g}{|dT|^2_g}.
\]
\end{lem}

\noi The proof is left to the reader.

We are interested in obtaining a certain convexity from this setup by making use of Theorem \ref{dit-triv}.  Unfortunately, Theorem \ref{dit-triv} does not apply to $\mathbf{H}(\fh, g) \to \bbL$, since $\sX \to \bbL$ is not a trivial family.  To remedy this problem we consider instead the trivial family $p_1 ^* E \to X \times \bbL \to \bbL$ with a non-trivial metric.  More precisely, define the plurisubharmonic functions 
\[
\psi : X \times \bbL \ni (x, \tau) \mapsto \max ( \log |T(x)|^2 - \Re \tau, 0) \quad \text{and} \quad \psi _{\tau}(x) := \psi (x,\tau).
\]
The functions $\psi$ is non-negative and depends only on $\Re \tau$, and the support of $\psi _{\tau}$ is $X- X_{\Re \tau}$.  We define the metrics 
\[
\fh ^{(p)} _{\tau} := e^{-p \psi _{\tau}} \fh
\]
and the Hilbert spaces 
\[
\sH ^{(p)} _{\tau} :=  \left \{ F \in H^0(X, \cO (E))\ ;\ ||F||^2_{\tau} := e^{-\Re \tau} \int _{X} \fh ^{(p)} _{\tau}(F,F) dV_g < +\infty \right \}.
\] 
Note that for each $F \in \sH^{(p)} _{\tau}$ one has 
\[
\lim _{p \to \infty} e^{-\Re \tau} \int _{X} \fh^{(p)} _{\tau} (F,F) dV_g = e^{-\Re \tau} \int _{X_{\Re \tau}} \!\!\!\!\!\!\!\! \fh (F,F) dV_g.
\]
Thus for large $p$ the Hilbert bundle $\sH ^{(p)} \to \bbL$ approximates $\mathbf{H} (\fh , g) \to \bbL$ in a certain sense, and Theorem \ref{dit-triv} can be applied to $\sH ^{(p)}$.

To carry out the approximation of $\mathbf{H}(\fh, g)$ by $\sH^{(p)}$ we shall use the following lemma, which is a special case of \cite[Lemma 3.4]{bl}.  

\begin{lem}\label{deal-with-p}
Let $\nu : (-\infty, 0) \to \bbR_+$ be an increasing function such that $\nu (t) \le e^t$ for all $t < 0$.  Then for $p > 1$, 
\begin{equation}\label{lim-inf}
\liminf _{t \to -\infty} e^{-t} \int _t ^0 e^{-p(s-t)} d\nu (s) \le \frac{2}{p-1}.
\end{equation}
\end{lem}

\medskip

We now make our way toward the proof of Theorem \ref{ot-vb-flat}.  

\begin{lem}\label{xi_g-is-bounded}
Let $\sigma \in H^0(Z, \sC^{\infty}(E))$ have compact support.  Then 
\[
\sup _{\tau \in \bbL} ||\xi_{\sigma}||^2_{\tau *} < +\infty.
\]
\end{lem}

\begin{proof}
Let $H \in \sH ^{(p)}_{\tau}$.  For each $x \in Z$ we can choose a small holomorphic disk $D_r(x) \subset X$ that is perpendicular (with respect to the metric $g$) to $Z$, lies in a neighborhood $U \subset X$ of $x$ such that $T$ is a component of a coordinate system in $U$, and one has $T (D_r(x)) = D_r(0)$.  If $\tau$ is far away from $-\infty$ we can choose $r$ independently of $\tau$, and if $\re \tau$ is so negative that $U$ is no longer contained in $X_{\re \tau}$ we choose $r = e^{\tau/2}$.  By the sub-mean value property we have 
\[
\fh (H(x), H(x)) \le C e^{- \re \tau}\int _{D_r(x)} \fh (H,H) \ii dT \wedge d\bar T.
\]
The factor $e^{-\re \tau}$ comes about because the area of this disk $D_r(x)$ with respect to the measure $e^{-\re \tau}\ii dT \wedge d\bar T$ is bounded above and away from $0$ uniformly in $\tau$.  Indeed, for $\tau$ sufficiently far from $-\infty$ this is clear, and for $\re \tau << 0$ the choice $r = e^{\re \tau/2}$ dictates that the area of $D_r(x)$ with respect to the measure $e^{-\re \tau} \ii dT \wedge d\bar T$ is $2\pi$.  The constant $C$, which of course depends on the metrics $\fh$ and $g$, depends only locally uniformly on $x$, so in particular if $x$ lies a relatively compact subset of $Z$ containing the support of $\sigma$ then we can choose $C$ uniform.  If we now integrate over a ball $\Delta(x)$ centered at $x$ and that is relatively compact in $Z$ then by Fubini we have 
\begin{eqnarray*}
\fh (H(x), H(x)) &\le& C'e^{-\re \tau} \int _{\Delta (x) \times D_r(x)} \fh (H,H) \omega_g ^{n-1} \wedge dT \wedge d\bar T \\
&=&  C' e^{-\re \tau} \int _{\Delta(x) \times D_r(x)} \fh (H,H) |dT|^2_gdV_g.
\end{eqnarray*}
Therefore for some constant $C_{\sigma}$ depending on $\fh$, $g$ and $\sigma$ and for each $x \in {\rm Support}({\sigma})\subset Z$ 
\[
\fh (H(x), H(x))  \le C_{\sigma} e^{- \Re \tau} \int _{X_{\Re \tau}} \fh (H,H)dV_g \le C_{\sigma} e^{- \Re \tau} \int _{X} \fh ^{(p)} _{\tau}(H,H)dV_g =  C_{\sigma} ||H||^2_{\tau}.
\]
Thus 
\[
||\xi_{\sigma}||^2_{\tau *} = \sup _{||H||_{\tau} = 1} \left | \int _Z \fh (H , {\sigma})  \frac{dA_g}{|dT|^2_{g}}\right |^2 \le C_{\sigma},
\]
as desired.
\end{proof}

\begin{lem}\label{convex-inc}
Let $\sigma \in H^0(Z, \sC^{\infty}(E))$ have compact support.  Then the function 
\[
\lambda_{\sigma}  : (-\infty, 0] \ni t \mapsto \log ||\xi _{\sigma}||^2_{t*}
\]
is non-decreasing.  In particular, 
\[
||\xi_{\sigma}||^2_{o*} \ge  ||\xi_{\sigma}||^2_{t^*}
\]
for all $t < 0$.
\end{lem}

\begin{proof}
By Theorem \ref{dit-triv} the function $\tau \mapsto \log ||\xi _{\sigma}||^2_{\tau*}$ is subharmonic on $\bbL$.  On the other hand,  $||\xi _{\sigma}||^2_{\tau*}$ depends only on $\Re \tau$, and thus $\lambda _{\sigma}$ is convex on $(-\infty, 0)$.  If $\lambda_{\sigma}$ decreases anywhere on $(-\infty, 0)$ then by convexity $\lim _{t \to -\infty} \lambda_{\sigma} = +\infty$, which contradicts Lemma \ref{xi_g-is-bounded}.
\end{proof}

\begin{thm}\label{lower-bound-for-xis}
Let $\sigma \in H^0(Z, \sC^{\infty}(E))$ have compact support. Then for each $\delta > 0$ there exists $p$ sufficiently large so that 
\[
\lim _{t \to - \infty}  ||\xi_{\sigma}||^2_{t*} \ge \frac{1}{\pi} \int _Z \fh (P\sigma, P\sigma) \frac{dA_g}{|dT|^2_g} -\delta.
\]
\end{thm}

Theorem \ref{lower-bound-for-xis} is proved using Lemmas \ref{fubini} and \ref{deal-with-p}, as well as the following lemma.  

\begin{lem}\label{bergman-runge}
Let $\sigma \in H^0(Z, \sC^{\infty}(E))$ have compact support.  Then for any $\delta > 0$ there exists $G \in H^0(X,\cO (E))$ that is holomorphic up to the boundary, such that 
\[
\int _Z  \fh( G - P\sigma, G-P\sigma)\frac{dA_{g}}{|dT|^2_{g}} \le \delta ^2 .
\]
\end{lem}

\begin{proof}
By $L^2$ approximation theorems (see, e.g., \cite[Theorem 5.6.2]{hormander-book}) we can find a section that is holomorphic on a neighborhood of $\overline{Z}$ and approximate $P\sigma$ uniformly on any compact subset of $Z$.  
If we take the compact subset to be sufficiently large then we also get approximation in $L^2$.  Finally, since the ambient manifold is Stein, the approximation on a neighborhood of $\overline{Z}$ can be extended to a neighborhood of $\overline{X}$, and therefore the extension will have finite $L^2$ norm on $X$.
\end{proof}

\begin{proof}[Proof of Theorem \ref{lower-bound-for-xis}]
If $P\sigma = 0$ then there is nothing to prove, so we assume $P\sigma \neq 0$.  Also to slightly simplify the already cumbersome notation, we set 
\[
||P\sigma|| := \left ( \int _Z \fh (P\sigma,P\sigma) \frac{dA_g}{|dT|^2_g} \right ) ^{1/2}.
\]

Fix $\ve > 0$ and let $\sigma$ and $G$ be as in Lemma \ref{bergman-runge} with $\delta = \ve ||P\sigma||$ there.  Then
\begin{eqnarray*}
\left | \int _Z \fh (G,P\sigma) \frac{dA_g}{|dT|^2_g} \right | &=& \left | \int _Z \fh (G- P\sigma,P\sigma) \frac{dA_g}{|dT|^2_g}  + ||P\sigma||^2 \right | \\
&\ge& ||P\sigma||^2 - ||P\sigma|| \left (\int _Z  \fh( G - P\sigma, G-P\sigma)\frac{dA_{g}}{|dT|^2_g} \right )^{1/2} \ge (1- \ve)||P\sigma||^2,
\end{eqnarray*}
and thus 
\begin{equation}\label{xi-est}
||\xi _\sigma||^2_{t^*} \ge \frac{1}{||G||^2_t} \left | \int _Z \fh (G ,P\sigma)\frac{dA_{g}}{|dT|^2_{g}}\right |^2 \ge \frac{(1 - \ve)^2 ||P\sigma||^4}{||G||^2_t}.
\end{equation}
Now, 
\begin{equation}\label{Gtnorm-eqn}
||G||^2_t = e^{-t}\int _{X_t} \fh (G,G)dV_{g} + e^{(p-1)t }\int _{X-X_t}e^{-p\log |T|^2}\fh(G,G)dV_{g}.
\end{equation}
By Lemma \ref{fubini} and our assumption that $P\sigma \neq 0$, we can choose $t << 0$ so that 
\[
e^{-t}\int _{X_t} \fh(G,G)dV_{g}\le  (1+ \ve)||P\sigma||^2.
\]
For the second integral on the right hand side of \eqref{Gtnorm-eqn} we have
\[
e^{(p-1)t}\int _{X-X_t} e^{-p\log |T|^2}\fh(G,G) dV_{g} \le \left ( \sup _X \fh(G,G)\right ) e^{(p-1)t} \int _{X-X_t} e^{-p \log |T|^2} dV_{g}.
\]
But 
\[
e^{(p-1)t} \int _{X-X_t} e^{-p\log |T|^2} dV_{g} \le C  e^{-t} \int _t ^0  e^{-p(s-t)}d\nu (s), 
\]
where $\nu (t) = \int _{X_t}dV_{g}$.  Since $\nu (t) \le C' e^t$, Lemma \ref{deal-with-p} implies that 
\[
e^{(p-1)t }\int _{X-X_t} e^{-p\log |T|^2}\fh(G,G) dV_g  \le \frac{C_o}{p-1}
\]
for a sequence of `$t$'s tending to $-\infty$.  Thus for sufficiently large $p$ and sufficiently negative $t$ 
\[
||G||^2_t  \le (1+ \ve)||P\sigma||^2 + \ve.
\]
From \eqref{xi-est} we have 
\[
||\xi _\sigma||^2_{t^*} \ge \frac{(1-\ve)^2||P\sigma||^4}{(1+\ve)||P\sigma||^2 + \ve}
\]
Choosing $\ve > 0$ small enough yields the desired estimate.
\end{proof}

By combining Lemma \ref{convex-inc} and Theorem \ref{lower-bound-for-xis} we obtain \eqref{key-estimate-ot-vb-flat}, and Theorem \ref{ot-vb-flat} is proved.
\qed

\subsubsection{\sc Grauert Duality:  End of the proof of Theorem \ref{OTMnc}}

\subsubsection*{\bf Homogeneous expansion of holomorphic functions}

\begin{prop}\label{homog-expansion}
Let $E \to X$ be a holomorphic vector bundle, let $H \to X$ be a holomorphic line bundle and denote by $\fp :H^* \to X$ the dual bundle.  Let $\sigma \in \Gamma _{\cO} (H^*, \fp ^*H^*)$ denote the diagonal section 
\[
\sigma (v) := (v,v).
\]
Then for any $s \in H^0(H^*, \cO (\fp ^* E))$ there exist sections $a_j \in H^0 (X, \cO(H^{\tensor j}\tensor E))$, $j=0,1,...$ ,  such that 
\[
s = \sum _{j=0} ^{\infty} (\fp ^* a_j)\tensor \sigma ^{\tensor j}.
\]
\end{prop}

\begin{proof}
Fix some $v \in H^*$ and a frame $\eta_1,...,\eta _r$ for $E \to X$ near $x = \fp v$.  Let $\xi \in H^*_v- \{0\}$ and let $t \xi$ be a typical point on the fiber $H_v$.  Writing our section as $s = f^{\mu} \fp ^* \eta_{\mu}$, we have the power series expansions 
\[
f^{\mu}(t\xi) = \sum _{j \ge 0} A^{\mu}_j(\xi) t^j, \quad 1 \le \mu \le r.
\]
If one starts with another $\tilde \xi \in H^*_v- \{0\}$ and $\tilde t \in \bbC$ such that $t \xi = \tilde t \tilde \xi$ then  
\[
\sum _{j \ge 0} A^{\mu}_j(\xi) t^j = f^{\mu}(t\xi) = f^{\mu}(\tilde t \tilde \xi) =  \sum _{j \ge 0} \tilde A^{\mu}_j(\tilde \xi) \tilde t^{j},
\]
from which it follows that 
\[
a_j(\fp \xi) := A^{\mu}_j(\xi) \eta _{\mu}\tensor \xi ^{\tensor - j}
\]
is well-defined, independent of $\xi$.  Thus $a_j \in \Gamma _{\cO} (X, E\tensor H^{\tensor j})$, and 
\[
\sum _{j=0} ^{\infty} \left ( \fp ^*a _j \tensor \sigma ^{\tensor j} \right ) (v) = \sum _{j=0} ^{\infty} A^{\mu}_j(v) (v)^{\tensor -j} \sigma (v)^{\tensor j} \tensor \fp ^* \eta_{\mu} = \sum _{j=0} ^{\infty} A^{\mu}_j(v)  \fp ^* \eta_{\mu} = f ^{\mu}\fp ^* \eta_{\mu},
\]
as required.
\end{proof}

\subsubsection*{\bf The unit disk bundle associated to an extension problem}\label{disk-bundle-geom}\ 

\medskip

We return to the setting of Theorem \ref{OTMnc}.  We have a Stein manifold $X$ and a possibly singular hypersurface $Z \subset X$.  Unlike the flat case studied in the previous section, the hypersurface $Z$ need not be cut out by a holomorphic function, much less a bounded one.  However, there is a line bundle $L_Z$ and a section $T \in H^0(X, \cO (L_Z))$ whose zero locus, counting multiplicity, is precisely $Z$.  

The more restrictive curvature hypothesis \eqref{curv-hyp-OTBL-vb} implies that the metric $\fg$ for $L_Z \to X$ has non-negative curvature.  In the total space of the dual bundle $L_Z^*$, with its dual metric $\fg ^*$, one can define the unit disk bundle 
\[
\sB (\fg) := \{ v \in L_Z ^*\ ;\ |v|^2_{\fg ^*} < 1 \}.
\]
A well-known observation of Grauert is that the vertical boundary 
\[
\di^{\rm vert}\sB (\fg) := \{ v \in L_Z^*\ ;\ |v|^2e^{\lambda} = 1\}
\]
of $\sB (\fg)$ is pseudoconvex if and only if $\fg$ has non-negative curvature.  The need for the pseudoconvexity of $\di^{\rm vert}\sB (\fg)$ is precisely the reason for assuming the stronger curvature condition \eqref{curv-hyp-OTBL-vb}.  Since the base $X$ is Stein, $\sB (\fg)$ is also Stein.

To the section $T$ we can associate a holomorphic function $\sT \in \cO (L_Z^*)$ defined by 
\[
\sT (v) = \left < v, T(\fp v)\right >,
\]
where $\fp : L_Z ^* \to X$ is the line-bundle projection.  On $\sB (\fg)$ we have the bound 
\[
|\sT (v)|^2 = |\left < v, T(\fp v)\right >|^2 = |v|^2 e^{\lambda} |T(\fp v)|^2e^{-\lambda} < 1.
\]
Note that the zero locus of of $\sT$ is 
\[
\sT^{-1}(0) =\cZ \cup \bbO _{L_Z^*},
\]
where $\cZ = \fp ^{-1} (Z)$ and $\bbO_{L_Z^*}$ is the zero section of $L_Z ^* \to X$.

We use the volume form  
\[
d\sV :=  \frac{\fp ^* \omega ^n}{n!} \wedge dd^c |v|^2_{\fg^*}
\]
on $\sB (\fg)$.  Its curvature, when viewed as a metric for $K_{\sB (\fg)}^*$, is 
\[
- \di \dbar \log d\sV = \fp ^*( {\rm Ricci}(\omega) - \Theta(\fg)).
\]
To make use of Theorem \ref{ot-vb-flat} we also need a K\"ahler metric.  Let us take the metric $\tilde g$ whose K\"ahler form is 
\[
\tilde \omega = \fp ^* \omega + c_o dd^c |v|^2_{\fg^*};
\]
the closed $(1,1)$-form $\tilde \omega$ is positive for small $c_o>0$.  The associated volume form is $dV_{\tilde g} = \det \tilde g$.

\subsubsection*{\bf End of the proof of Theorem \ref{OTMnc}}
We can pull back the vector bundle $E \to X$ and the line bundle $L_Z ^* \to X$ to the total space $L_Z ^*$ of $L_Z^* \to X$ via $\fp$ to define the vector bundles 
\[
E_m := \fp ^* (E \tensor L_Z ^{*\tensor m}) \to L_Z^*, \quad m \in \bbN.
\]
Let $\delta > 0$ be such that \eqref{curv-hyp-OTBL-vb} holds.  Letting $m$ be the unique positive integer such that 
\[
\delta _o := \delta +1 - m \in (0,1],
\]
we equip the vector bundle $E_m \to \sB(\fg)$ with the singular Hermitian metric 
\begin{equation}\label{db-metric}
\fh _m := (|\sigma |^2_{\tilde \fg})^{-(1-\delta_o)} \frac{d\sV}{\det \tilde g} \tilde \fg ^{\tensor m} \tensor \fh,
\end{equation}
where $\tilde \fg := \fp ^* \fg$ and $\sigma \in H^0(L_Z^*, \cO (\fp ^* L_Z^*))$ is the diagonal section (see Proposition \ref{homog-expansion}).  

We would like to apply Theorem \ref{ot-vb-flat} with the metric $\fh _m$, but we must be slightly more careful; this metric is not smooth.  However, it has a very simple regularization, namely 
\[
\fh _m ^{(\ve)} := (|\sigma |^2_{\tilde \fg}+ \ve^2 )^{-(1-\delta_o)} \frac{d\sV}{\det \tilde g} \tilde \fg ^{\tensor m} \tensor \fh.
\]
This regularization loses some positivity, but since we are on a relatively compact domain in a Stein manifold, we can regain this positivity using a bounded strictly plurisubharmonic function.  The details, which are standard, are omitted for brevity, and we shall treat $\fh_m$ like a smooth metric.

Under the hypotheses of Theorem \ref{OTMnc} 
the metric 
\[
\fh _m \tensor \det \tilde g
\]
for $E_m \tensor K^*_{\sB (\fg)}$ is Nakano-semipositive.  By Theorem \ref{ot-vb-flat} we obtain the following lemma.

\begin{lem}\label{ot-vb-db}
Let $\ff \in H^0(\sT ^{-1}(0), \cO (E_m))$ satisfy
\[
\int _{\sT^{-1}(0)} \fh _m (\ff, \ff) \frac{dA_{\tilde g}}{|d\sT|^2_{\tilde g}} < +\infty.
\]
Then there exists $\fF \in H^0(\sB(\fg), \cO (E_m))$ such that 
\[
\fF |_{\sT^{-1}(0)} = \ff \quad \text{and} \quad \int _{\sB(\fg)} \fh _m (\fF, \fF) dV_{\tilde g} \le \pi \int _{\cZ} \fh _m (\ff, \ff) \frac{dA_{\tilde g}}{|d\sT|^2_{\tilde g}}.
\]
\end{lem}

If we apply Lemma \ref{ot-vb-db} to the section 
\[
\ff := \sigma ^{\tensor m} \tensor (\pi ^* f)
\]
we find a section $\fF \in H^0(\sB(\fg), \cO (E_m))$ such that $\fF |_{\sT ^{-1}(0)} = \sigma ^{\tensor m} \tensor \fp ^* \ff$ and
\begin{eqnarray*}
\int _{\sB(\fg)} \fh _m (\fF, \fF) dV_{\tilde g} &\le& \pi \int _{\cZ} \fh _m (\ff, \ff) \frac{dA_{\tilde g}}{|d\sT|^2_{\tilde g}}\\
&=& \pi \int _{\cZ} \fp ^* (\fh (f,f))( |\sigma |^2_{\fg ^*})^{m-\delta _o} dd^c( |\sigma|^2_{\tilde \fg^*}) \wedge \fp ^* \left (\frac{dA_g}{|dT|^2_{g, \fg}}\right )\\
 &=& \pi \int _{\cZ} \fp ^* (\fh (f,f))( |\sigma |^2 _{\tilde \fg^*})^{\delta -1} dd^c( |\sigma|^2_{\tilde \fg^*}) \wedge \fp ^* \left (\frac{dA_g}{|dT|^2_{g,\fg}}\right ).
\end{eqnarray*}
The section $\fF$ vanishes to order $m+1$ along $\bbO _{L_Z^*}$, and therefore 
\[
\fF = \sigma ^{\tensor m} \tensor \Phi
\]
for some section $\Phi \in H^0(\sB (\fg), \fp ^* E)$ which evidently satisfies 
\begin{eqnarray*}
&& \int _{\sB (\fg)} (\fp ^*\fh) (\Phi , \Phi)( |\sigma |^2_{\tilde \fg^*})^{m+1-\delta _o} dd^c( |\sigma|^2_{\tilde \fg^*}) \wedge \fp ^* dV_g  \\
&=& \int _{\sB (\fg)} (\fp ^*\fh) (\Phi , \Phi)( |\sigma |^2_{\tilde \fg^*})^{\delta} dd^c( |\sigma|^2_{\tilde \fg^*}) \wedge \fp ^* dV_g \\ &\le & \pi \int _{\cZ} \fp ^* (\fh (f,f))( |\sigma |^2_{\tilde \fg^*})^{\delta -1} dd^c( |\sigma|^2_{\tilde \fg^*}) \wedge \fp ^* \left (\frac{dA_g}{|dT|^2_{g,\fg}}\right ).
\end{eqnarray*}
Integration along the fibers of the disk bundle $\cZ$ shows that 
\[
\int _{\cZ} \fp ^* (\fh (f,f))( |\sigma |^2_{\tilde \fg^*})^{\delta -1} dd^c( |\sigma|^2_{\tilde \fg^*}) \wedge \fp ^* \left (\frac{dA_g}{|dT|^2_{g,\fg}}\right ) = \frac{1}{\delta} \int _Z \fh (f,f) \frac{dA_g}{|dT|^2_{g,\fg}}.
\]
Now, by Proposition \ref{homog-expansion} the section $\Phi$ has the (convergent) series expansion 
\[
\Phi = \fp ^* F  + \sum _{i \ge 1} \sigma ^{\tensor i} \fp ^* a_i.
\]
In particular, $F$ is an extension of $f$.  Moreover, by orthogonality we have
\begin{eqnarray*}
&& \int _{\sB (\fg)} (\fp ^*\fh) (\Phi , \Phi)( |\sigma |^2_{\tilde \fg^*})^{\delta} dd^c( |\sigma|^2_{\tilde \fg^*}) \wedge \fp ^* dV_g \\
&=& \frac{1}{\delta +1} \int _{X} \fh (F,F)  dV_g  + \sum _{i\ge 1}\frac{1}{\delta +i} \int _{X} \fh (a_i,a_i)  dV_g\\
&\ge &\frac{1}{\delta +1} \int _{X} \fh (F,F)  dV_g.
\end{eqnarray*}
Hence the extension $F$ of $f$ satisfies 
\[
\int _X \fh (F,F) dV_g \le \frac{\pi (1+\delta)}{\delta} \int _Z \fh (f,f) \frac{dA_g}{|dT|^2_{g,\fg}}.
\]
The proof of Theorem \ref{OTMnc} under the stronger hypothesis \eqref{curv-hyp-OTBL-vb} is therefore complete.
\qed

\begin{rmk}
A slight modification of Lemma \ref{deal-with-p} discovered by Albesiano in \cite{roberto} can be used here to yield Theorem \ref{OTM} with the curvature hypotheses 
\begin{equation}
\ii \Theta (\fh) +{\rm Ricci}(g) \ge_{\rm Nak} \delta \Theta (\fg) \tensor {\rm Id}_{E} \quad\text{and} \quad \ii \Theta (\fh) \ge_{\rm Nak} (1+\delta) \ii \Theta (\fg) \tensor {\rm Id}_{E}.
\end{equation}
This result was also achieved by Nguyen and Wang \cite{nw}, by rather different methods that, though more complicated than Albesiano's approach, seem to have farther-reaching consequences.

Though these results constitute an improvement on \eqref{curv-hyp-OTBL-vb}, they still do not reach \eqref{curv-hyp-OTnc}.  It remains an open question whether the degeneration technique of Berndtsson and Lempert can somehow be modified to prove the full Theorem \ref{OTM}.  Moreover, thus far the method has not been able to achieve other extension theorems.  For example, it is not known whether a proof of Theorem \ref{gz-thm} of Guan-Zhou is possible by the method of Berndtsson and Lempert.
\red
\end{rmk}

\section{Attempting to prove Theorem \ref{OTMnc} by passing to $\cO _E(1) \to \bbP(E^*)$}\label{OT-Grif-pf-section}

A standard technique for treating higher rank vector bundles is to reduce them to line bundles by the ideas discussed in Paragraph \ref{O1-of-E-par}.  There remain difficult questions regarding how far this technique can reach, with the most famous problem being the Griffiths conjecture that ampleness in the sense of Hartshorne is equivalent to the existence of a metric of Griffiths-positive curvature.  

In view of Proposition \ref{section-id-prop} it is natural to inquire whether a proof Theorem \ref{OTM} could be undertaken by passing to the projectivization.  It is well-known and easy to check that the metric $\fh$ is Griffiths-positive if and only if the metric $e^{-\vp_{\fh}}$ for $\pi : \cO (1) \to \bbP(E^*)$ is positive.  However, while there is a correspondence between sections of $E$ and those of $\cO _E(1) \to \bbP(E^*)$, this correspondence sends sections of $K_X \tensor E$ to sections of $\pi ^* K_X \tensor \cO _E(1)$.  To obtain sections of $K_{\bbP(E^*)} \tensor L \to \bbP(E^*)$ for an appropriate line bundle $L \to \bbP(E^*)$, one must find the relationship between $\cO_{\bbP(E^*)} (\pi ^* K_X)$ and $\cO_{\bbP(E^*)}(K_{\bbP(E^*)}$.  The key is the formula 
\[
K_{\bbP(E^*)}= \cO_E (-r) \tensor \pi ^* (K_X \tensor \det E)
\]
(see \cite[Proposition 3.6.20]{kob}).  We conclude that 
\[
\pi ^* K_X \tensor \cO _E(1) = K_{\bbP(E^*)} \tensor \cO _E(r+1) \tensor \pi ^*(\det E^*).
\]

If we use only that $\bbP(E^*)$ is a complex manifold, and do not try to exploit its special geometric features, then to apply Theorem \ref{OTM} in the rank-1 case we need to assume at least that the line bundle $\cO _E(r+1) \tensor \pi ^*(\det E^*) \to \bbP(E^*)$ is pseudo-effective (and of course, we would also need to overcome the normal bundle of $\pi ^{-1} (Z)$, though we will not discuss this issue here).

Consider the line bundle 
\[
\cO_E (r+1) \tensor \det E^*.
\]
The positivity of this line bundle follows if we show that  the ``$\bbQ$-vector bundle" 
\[
F := (\det E^*)^{1/(r+1)} \tensor E
\]
is Griffiths-positive.  (The converse is a case of the aforementioned Griffiths conjecture.)  But even if we assume only that $\cO_E (1) \tensor (\det E^*)^{1/(r+1)}$ is (semi)positive, a beautiful result of Berndtsson \cite[Theorem 7.1]{bo-annals} implies that the vector bundle 
\[
F \tensor \det F = \left ((\det E^*)^{1/(r+1)} \tensor E \right ) \tensor (\det E^*)^{r/(r+1)} \tensor \det E = E
\]
is Nakano-(semi)positive.

In order to give a proof of Theorem \ref{OTM} by passing to $\cO _E(1) \to \bbP(E^*)$, we need the following converse:  if the metric for $E$ is Nakano-(semi)positive then the induced metric for $\cO _E(r+1) \tensor \det E^*$ is (semi)positive.  Unfortunately, such a result is not true in general.  The following example presents a Nakano-positive holomorphic Hermitian vector bundle $E$ such that the induced metric on $\cO_E(r+1) \tensor \det E^*$ is not semi-positive.

\begin{ex}
Let $L^i  \to X$ be holomorphic line bundles with smooth Hermitian metrics $\fh^i$ and consider the totally split holomorphic Hermitian vector bundle let 
\[
E = L^1 \oplus \cdots \oplus L^r \quad \text{ with the induced metric }\fh := \fh ^1 \oplus \cdots \oplus \fh ^r.
\]
A holomorphic section $s \in \cO (E)_x$ has a decomposition $\sigma = (\sigma^1,...,\sigma^r)$ with $\sigma^i \in \cO (L^{i})_x$.

We begin by observing that for such totally split $(E,\fh)$ the following are equivalent.
\begin{enumerate}[i.]
\item $(E, \fh)$ is Nakano positive.
\item Each $(L^i,\fh^i)$ is positive.
\item $(E,\fh)$ is Griffiths positive.
\end{enumerate}
The implications (i) $\Rightarrow$ (iii) $\Rightarrow$ (ii) are straightforward. Now, 
\begin{equation}\label{nak-split-cond-eq}
\sum _{\mu,\nu=1} ^N \fh (\ii \Theta (\fh)_{\xi _{\mu}\bar \xi _{\nu}}\sigma_{\mu},\sigma_{\nu}) = \sum _{\mu,\nu =1} ^N \sum _{i=1} ^r \Theta (\fh ^i)_{\xi _{\mu}\bar \xi _{\nu}} \fh ^i (\sigma_{\mu}^i, \sigma_{\nu}^i).
\end{equation}
For each $i$ the matrix $\left (\fh ^i (\sigma_{\mu}^i, \sigma_{\nu}^i) \right )_{\mu, \nu}$ is Hermitian, so there is a unitary matrix $U(i)$ such that
\[
\fh ^i (U(i)_{\mu} ^{\alpha} \sigma_{\alpha}^i, U(i)_{\nu}^{\beta} \sigma_{\beta}^i) = |a^i_{\mu}|^2 \delta _{\mu \bar \nu};
\]
the right hand side must be non-negative because $\fh ^i$ is a metric.  Therefore the right hand side of \eqref{nak-split-cond-eq} is positive if each of the matrices $\left ( \Theta (\fh ^i)_{\xi _{\mu}\bar \xi _{\nu}} \right )_{\mu,\nu}$ is positive definite, so that (ii)  $\Rightarrow$ (i).

We shall now show that for some choice of metrics $\fh ^1,..., \fh ^r$ the metric $\tilde \fh _r$ induced on $\cO_E(r+1)\tensor \det E^*$ does not have positive curvature.  

First, the metric induced on $\cO _E(-1)$ by $\fh$ is given by 
\[
|v|^2e^{\vp_{\fh}} = \fh ^*(v,v);
\]
see Paragraph \ref{O1-par}.  Thus the curvature $\ii \di \dbar \vp _{\fh}$ of the dual metric $e^{-\vp _{\fh}}$ for $\cO_E(1)$ is the negative of the curvature of the metric $e^{\vp _{\fh}}$, so 
\[
\ii \di \dbar \vp _{\fh} = - \ii \di \dbar \log |\sigma|^2e^{-\vp_{\fh}} = \ii \di \dbar \log \fh ^*(s,s),
\]
where $\sigma$ is a local frame for $\cO_E(1)$, and $s= (s^1,...,s^r)$ is the corresponding section of $E^*$ associated to the dual frame $\sigma ^*$ by the formula $\hat s = \sigma ^*$; see Proposition \ref{section-id-prop}, noting that the latter proposition is applied to the vector bundle $E^*$, not $E$.  (The formula for $\ii \di \dbar \vp _{\fh}$ is independent of the choice of $s$ representing a section of $\cO_E(-1)$, because another choice $\tilde s$ must be of the form $\tilde s = \lambda s$ for some nowhere-vanishing holomorphic function on the domain of $s$.)

We compute that
\begin{eqnarray*}
&& \ii \di\dbar \log \fh^* (s,s) = \ii \di \dbar \log \sum _{i=1} ^r \fh^{i*} (s^i,s^i) \\
&& \quad = \frac{\sum _{i=1} ^r \fh^{i*} (-\ii \Theta (\fh^{i*})s^i,s^i)}{\sum _{i=1} ^r \fh^{i*} (s^i,s^i)}  \\
&& \quad \quad +\frac{\left ( \sum _{i=1} ^r \fh^{i*} (s^i,s^i) \right ) \left (\fh^{i*} (\nabla s^i, \nabla s^i) \right ) - \ii \left (\sum _{i=1} ^r \fh^{i*} (\nabla s^i,s^i)\right ) \wedge \left (\sum _{i=1} ^r \fh^{i*} (s^i,\nabla s^i) \right )}{\left (\sum _{i=1} ^r \fh^{i*} (s^i,s^i)\right )^2}\\
&& \quad \ge \frac{\sum _{i=1} ^r \fh^{i*} (-\ii \Theta (\fh^{i*})s^i,s^i)}{\sum _{i=1} ^r \fh^{i*} (s^i,s^i)} = \frac{\sum _{i=1} ^r \ii \Theta (\fh^i) \fh^{i*} (s^i,s^i)}{\sum _{i=1} ^r \fh^{i*} (s^i,s^i)}.
\end{eqnarray*}
The inequality is sharp, as can be seen using sections with zero covariant derivative at a point.  Next 
\[
\ii \di \dbar \log \det \fh = - \sum _{i=1} ^r \ii \Theta (\fh^i).
\]
Therefore the curvature $\Theta(\hat \fh)$ of the metric $\hat \fh$ induced on $\cO _E(r+1) \tensor \det E^*$ by the split metric $\fh = \fh^1 \oplus ... \oplus \fh ^r$ for $E$ is positive if and only if for every collections of $r$ local holomorphic sections $s^i$ of $L^i$, $1 \le i \le r$ at least one of which is non-zero, one has 
\begin{equation}\label{OT-curv-0-ex1}
(r+1) \frac{\sum _{i=1} ^r \ii \Theta (\fh^i) \fh^{i*} (s^i,s^i)}{\sum _{i=1} ^r \fh^{i*} (s^i,s^i)} - \sum _{i=1} ^r \ii \Theta (\fh^i) \ge 0.
\end{equation}
The inequality \eqref{OT-curv-0-ex1} holds for all sections $(s^1,...,s^r)$ at least one of which is non-zero if and only if for every $\xi \in T^{1,0}_{X}$ and for all $\lambda _1,..., \lambda _r \in [0,1]$ such that $\lambda _1 +\cdots + \lambda _r = 1$ one has 
\begin{equation}\label{OT-curv-0-ex2}
\sum_i \lambda _i \Theta (\fh^i)_{\xi \bar \xi} \ge \frac{1}{r+1} \left (\Theta (\fh^1)_{\xi \bar \xi}+ \cdots + \Theta (\fh^r)_{\xi\bar \xi}\right )
\end{equation}
i.e., every convex combination of the Hermitian $(1,1)$-forms $\Theta (\fh^1), ... ,\Theta (\fh^r)$ is bounded below by the $(1,1)$-form $\frac{1}{r+1} \left (\Theta (\fh^1)+ \cdots + \Theta (\fh^r)\right )$.

However, there exist positively curved metrics $\fh ^1,...,\fh^r$ (so in particular, $E$ is Nakano positive) for which this is not the case.  Indeed, \eqref{OT-curv-0-ex2} implies that 
\begin{equation}\label{outcome-ex1}
\Theta (\fh^1)_{\xi \bar \xi}+ \cdots + \Theta (\fh^r)_{\xi \bar \xi} \ge 0 \quad \text{and} \quad \min _i  \Theta (\fh^i)_{\xi \bar \xi} \ge  \frac{1}{r} \max _i \Theta (\fh ^i)_{\xi \bar \xi},
\end{equation}
and we can clearly choose metrics $\fh^1,..., \fh ^r$ with positive curvature such that the second of these conditions fails.
\red
\end{ex}

\end{document}